\documentclass[a4paper,abstracton]{scrartcl}

\usepackage{amsmath, amssymb, amsthm}
\usepackage[english]{babel}
\usepackage{caption}
\usepackage{enumitem}
\usepackage{wrapfig}
\usepackage{tikz}
\usepackage{floatrow}

\numberwithin{equation}{section}

\newcommand{\eg}{e.g.,\ }

\newcommand{\R}{\mathbb{R}}
\newcommand{\Z}{\mathbb{Z}}
\newcommand{\N}{\mathbb{N}}

\newcommand{\del}{\partial}
\newcommand{\overbar}[1]{\mkern 1.5mu\overline{\mkern-1.5mu#1\mkern-1.5mu}\mkern 1.5mu}
\newcommand{\object}{\mathcal{O}}

\newcommand{\beq}{\begin{equation*}}
\newcommand{\eeq}{\end{equation*}}
\newcommand{\beqn}{\begin{equation}}
\newcommand{\eeqn}{\end{equation}}
\newcommand{\beqa}{\begin{eqnarray*}}
\newcommand{\eeqa}{\end{eqnarray*}}
\newcommand{\beqan}{\begin{eqnarray}}
\newcommand{\eeqan}{\end{eqnarray}}

\renewcommand{\d}{{\,\mathrm{d}}}

\newcommand{\LM}[1]{\hbox{\vrule width.2pt \vbox to#1pt{\vfill \hrule width#1pt height.2pt}}}
\newcommand{\LL}{{\mathchoice{\,\LM7\,}{\,\LM7\,}{\,\LM5\,}{\,\LM{3.35}\,}}}
\def\underbar#1{\underline{#1}}
\def\etal{\emph{et al.\,}}

\theoremstyle{plain}
\newtheorem{thrm}{Theorem}[section]
\newtheorem{lmm}[thrm]{Lemma}

\newtheorem{prpstn}[thrm]{Proposition}

\theoremstyle{definition}

\begin{document}

\title{A $BV$ Functional and its Relaxation for Joint Motion Estimation and Image Sequence Recovery}

\author{Sergio Conti, Janusz Ginster, and Martin Rumpf}\date{}
\maketitle

\begin{abstract}
The estimation of motion in an image sequence is a fundamental task in image processing.
Frequently, the image sequence is corrupted by noise and one simultaneously asks for the underlying motion field and a restored  
sequence.
In smoothly shaded regions of the restored image sequence
the brightness constancy assumption along motion paths leads to a pointwise differential condition on the motion field.
At object boundaries which are edge discontinuities both for the image intensity and for the motion field this condition is no longer well defined. 
In this paper a total-variation type functional is discussed for joint image restoration and motion estimation.
This functional turns out not to be     lower semicontinuous, 
and in particular fine-scale oscillations may appear around edges.
By the general theory of vector valued
$BV$ functionals its relaxation leads to the appearance of
 a singular part of the energy density, which can be determined by the solution
of a local minimization problem at edges.
 Based on bounds for the singular part of the energy  
and under appropriate assumptions on the local intensity variation one can exclude the existence of microstructures
and obtain a model  well-suited for simultaneous image restoration and motion estimation.
Indeed, the relaxed model incorporates a generalized variational formulation of the brightness constancy assumption.  
The analytical findings are related to ambiguity problems in motion estimation such as the proper distinction
between  foreground and background motion at object edges.
\end{abstract}

\section{Introduction}\label{sec:intro}
In computer vision, the accurate computation of motion fields in
image sequences---frequently called optimal flow estimation---is a long standing problem, which has been addressed
extensively. For a general overview on optical flow estimation we refer to the survey by Fleet and Weiss \cite{FlWe06}. 
We consider an image sequence given via 
a grey value map 
\beq
u: (0,T) \times \Omega \to \R \,;\quad (t,x)\mapsto u(t,x)
\eeq
on a space time domain $D:=(0,T) \times \Omega$, 
where $\Omega$ is a bounded Lipschitz domain in $\R^d$ for $d=1,2,3$.
To begin with, we suppose that motion is reflected by the image sequence and
that image points
move according to a velocity field $v: D \to \R^d$.
Hence, constancy of grey values $u(t,x(t))$ along motion trajectories $t\mapsto x(t)$ with $\dot x(t) = v(t,x(t))$
leads to the transport equation
\beqn\label{eq:BCCE}
0= \frac{\d}{\d t} u(t,x(t)) = \partial_t u(t,x) + \nabla_x u(t,x(t)) \cdot v(t,x(t))
\eeqn
as a constraint equation for the unknown velocity field $v$. 
This constraint equation is generally known as the brightness constancy constraint and 
for the space time motion field $w = (1,v)$ it can be rewritten as $\nabla u(t,x) \cdot w(t,x) = 0$.
Here and in what follows $\nabla = (\partial_t, \nabla_x)$ denotes the space time gradient.
This condition gives us pointwise one constraint for $d$ unknown velocity components.
Indeed, only the component of the velocity orthogonal to isolines of the grey value can be 
computed from equation \eqref{eq:BCCE}, which leads to an illposed problem known 
as the aperture problem. 
Nagel and Otte \cite{NaOt94} and Tristanelli \cite{Tri95} suggested to consider second derivatives, i.e.~$\frac{d}{dt} \nabla_x u(t,x(t)) = 0$ along motion trajectories $t\mapsto x(t)$
leads to
$\del_t \nabla_x u + \nabla_x^2 u \, v = 0$ so that, if $(\nabla_x^2 u)(t,x)$ is invertible, $v(t,x)$ can be computed.
Similarly, in more geometric terms the motion field can be described via temporal 
variations of the shape operator on level sets as proposed by Guichard \cite{Gu98}.
Since second derivatives are involved, these pointwise approaches are vulnerable to noise and hence of difficult practical usability.
Based on the assumption that the image intensity $u$ varies on a finer scale than $v$,
one might assume $v$ to be locally constant and accumulate locally different constraint equations to estimate $v$. This dates back to the early work by Lucas
and Kanade \cite{LuKa81,WeBrSc03} 
or the {\em structure tensor approach} \cite{BiGr88}, which minimizes the local energy functional
$ \int_D \omega(t-s,x-z) (\nabla u(s,z) \cdot v(t,x))^2 \d z\d s$
for a local window function $\omega(\cdot,\cdot)$. 

This paper aims at consistently treating the general case with basically two different types of representations of motion in image sequences: \\[-2ex]
\begin{itemize}
\item[-] mostly smooth motion visible via spatial variations of object shading and texture and their transport in time,\\[-2.5ex]
\item[-]  motion represented by moving object edges, frequently characterized by discontinuities in the motion velocity apparent at edges of moving objects.  \\[-2.5ex]
\end{itemize}
The local approaches mentioned above are able to estimate the first type of motion and offer relatively high robustness with respect to noise but in general they do not lead to dense flow fields and fail to identify motion information concentrated on edges.
Global variational approaches were initiated by the work of Horn and Schunck \cite{HoSc81}.
They considered minimizers of the energy functional $\int_D |\nabla u \cdot w|^2 + \alpha |\nabla v|^2 \d x \d t$
implicitly assuming the optical flow field $v$ to be smooth. 
A rigorous numerical analysis for a finite difference discretization of the Horn--Schunck approach was performed by Le Tarnec \etal \cite{LeDe14}.
Nagel and Enkelmann \cite{NaEn86} proposed to use an anisotropic regularization term with a smaller penalization for  variations of $v$ in normal direction across edges. With a focus on real world applications,
Weickert \etal \cite{WeBrSc03} proposed a combination of local flow estimation and global 
variational techniques to combine the benefits of robustness and dense field representation,
respectively. 
For a detailed analysis of the occlusion problem associated with the estimation of object motion we refer to the joint approach for motion estimation and segmentation by Kanglin and Lorenz \cite{KaLo12}.
Ito  \cite{It05} suggested to treat the optical flow estimation in terms of an optimal control formulation. 
Brune \etal \cite{BrMa09} used the optimal control paradigm to estimate intensity and motion edges in image sequences.

Before discussing total variation type approaches for motion estimation ---to which our method belongs--- we
investigate a basic but already characteristic optical flow problem.\\[-2ex]
\emph{\paragraph{A simple model problem.} 
Suppose an object $\object$ with a shading or texture intensity map $u_1$ is moving with spatially constant 
velocity $v_1$ on a background with shading and texture intensity map $u_2$, which is itself moving with  constant velocity $v_2$. Thus, the observed image sequence is given by
\beqn \label{eq:imagefunc}
u(t,x) = \chi_\object(x-t v_1) u_1(x-t v_1) + (1-\chi_\object(x-t v_1)) u_2(x-t v_2)\,,
\eeqn
where $\chi_\object$ denotes the characteristic function of the object domain $\object$.
Trying to retrieve the object and the two velocities $v_1, \, v_2$ from the 
image sequence one observes the following:\\[-2ex]
\begin{itemize}
\item[(i)] If both image intensities $u_1$ and $u_2$ are constant the role of foreground and background can be flipped, i.e.~either the object $\object$  or its complement $D\setminus \mathcal{O}$ is moving with speed $v=v_1$ , whereas the background velocity $v_2$ obviously cannot be determined. \\[-2.5ex]
\item[(ii)] If both  shading or texture intensity maps 
$u_1$ and $u_2$ are smooth and locally allow the computation of $v_1$ and $v_2$ from \eqref{eq:BCCE}, the decision on foreground and background is associated with the consistency of one of the velocities $v_i$ $( i=1,2 )$ with the motion of the interface. Thereby, consistency  is expressed in terms of the singular counterpart $n \cdot w_i =0$ of \eqref{eq:BCCE}, where $w_i = (1,v_i)$ and $n$ denotes the space time normal on the interface, which coincides with the jump set $J_u$ of the image function $u$ from 
\eqref{eq:imagefunc}. \\[-2.5ex]
\item[(iii)] If  neither $u_1$ nor $u_2$ is consistent with the motion of the jump set $J_u$ then the object is  undergoing a more complex evolution than just a rigid motion, \eg growth or shrinkage.
\end{itemize} 
}
\noindent
In the general case, beyond this simple model problem, at each point on the jump set
 $J_u$ of the space time intensity map $u$ one should compare  the values of the velocity  $v^+$ and $v^-$ on
the two sides
with the space time interface normal $n$ on $J_u$ to decide on the actual local motion pattern. 
\emph{Hence, we are interested in a  variational approach which explicitly incorporates this local consistency test, where the motion data $v^+$ and $v^-$ on the two sides are determined  either via  local shading
 or texture data or from a global relaxation principle taking into account far field motion data.}

The space of functions of bounded variation ($BV$) allows to describe configurations with singularities of codimension one, i.e.~edge-type jump sets. Total variation regularization was first introduced in image processing by Rudin, Osher, 
and Fatemi \cite{RuOsFa92}.
Cohen \cite{Co93} proposed to replace the usual quadratic regularization in the above motion estimation approaches by a $BV$ type regularization $\int_\Omega |\nabla v| \d y$. A more general  
convex regularizer with linear growth was investigated by Schn\"orr and Weickert \cite{WeSchn01b}.
A broader comparison of different regularization techniques in imaging and a discussion of suitable quasi-convex functionals was given by Hinterberger \etal \cite{HiScScWe01}. 
In particular they considered a $W^{1,p}$- approximation of $BV$ type functionals.
Papenberg \etal \cite{PaBrBr06}
investigated  a $TV$ regularization of the motion field and an optical flow constraint involving higher order gradients. 
In their pioneering paper Chambolle and Pock \cite{ChPo11} suggested a duality approach for nonsmooth convex optimization problems  
in $BV$ and discussed as one application the $TV$ motion estimation problem.  
An improvement of the original ansatz was suggested by Wedel \etal \cite{WePo09} and
an efficient implementation of this primal dual optimization approach to optical flow estimation 
was presented recently by S\'anchez \etal \cite{SaMe13}.

The approach by Aubert and Kornprobst  \cite{AuKo99} is closely related to ours.
They considered for $d=2$ the energy functional 
\beqn\label{eq:AuDeKo}
E[v] = \int_D |\nabla u \cdot (1,v)| + \phi(\nabla v)  + {\alpha_v} |v|^2 \d x \d t\,,
\eeqn
extended to the space of velocity fields  of bounded variation,
where $u$ is a fixed image intensity map in $L^\infty \cap SBV$ and $\phi:\R^2 \to \R$ a function with linear growth. 
They discussed the following fundamental problem.
On the jump set $J_u$ of $u$, representing the space time edge surfaces, the singular part of the gradient $D u$ 
is only a Radon measure and in the generic case of moving objects one expects a significant overlap of
$J_u$ with the jump set $J_v$ of the motion field $v$. Thus, it is unclear how to define $D u \cdot w$.
This is indeed a recasting of the above observation (ii) in the context of the theory of functions with bounded variation.
Aubert and Kornprobst considered a locally averaged evaluation of the motion field, and they finally studied 
the relaxation of the above functional in $BV$.
In the case of  Lipschitz continuous image sequences, Aubert  \emph{et al.} \cite{AuDeKo99} considered the numerical 
approximation of the above total variation functional based on a duality approach and a suitable approximation
of $\phi$ with a sequence of functionals of quadratic growth.

Frequently the image sequence is corrupted by noise that one wishes to remove. 
On the theoretical side, we cannot assume that the input image intensity is already in 
$BV$. We study here a $BV$ approach for 
the joint reconstruction of non smooth space time intensity maps $u$ 
and the underlying non smooth optical flow fields $v$. It differs
 from the ansatz by Aubert and Kornprobst  \cite{AuKo99} in that here \emph{both} fields are reconstructed simultaneously.
We investigate in particular the relation to the above-discussed fundamental observations (i)---(iii) on motion estimation
and show that the $BV$ approach naturally incorporates a local analysis of the motion pattern in the vicinity of the jump set $J_u$. 
Depending on the data, minimizing sequences may develop 
 undesirable small-scale oscillations around interfaces.  
Analytically this means that  the functional is not lower semicontinuous.
The theory of relaxation permits to replace the functional by its lower
semicontinuous envelope, thereby eliminating the fine-scale oscillations
from the kinematics, but still incorporating their averaged effect 
in the energetics. The key technical ingredient we use here is
 a general result on the relaxation of variational problems on
vector-valued $BV$ functions by M\"uller and Fonseca \cite{FoM93} and earlier work on relaxation on $BV$ functions by Ambrosio and Dal Maso \cite{AmbrosioDalmaso1992} and Aviles and Giga \cite{AvilesGiga1991}.
It leads to the relaxed functional presented in (\ref{eqfstar}) below.
The ambiguity close to the jump set is then resolved by minimizing locally a
suitable  microscopic problem, which in turn leads to a selection of 
 the relevant local  motion pattern, see (\ref{eq:definitionK}) below.
Using upper and lower bounds on the relaxed energy we can show that under suitable (implicit) assumptions on 
the image intensity map $u$ and the motion field $v$ such microscopic oscillations can be ruled out.
 
The advantages of the joint estimation of intensity $u$ and motion field $v$ are the following:\\[-2.5ex]
\begin{itemize}
\item[-]  A reliable segmentation of moving objects via the non smooth intensity map helps to estimate their motion.
\item[-] Given the motion field, the brightness constancy assumption along motion paths significantly improves
the denoising of the image sequence or even the restoration of missing frames.
\item[-] Reliable motion detection also poses an important cue for object detection and recognition. 
\end{itemize}
Thus, joint approaches which simultaneously estimate the motion field, segment objects and denoise the image sequence
are particularly appealing. Advances in this direction were
investigated in \cite{ScStYu07,OdBo98,CASELLES1996,MEMIN1998,PARAGIOS2000}. 
A first approach which relates optical flow estimation to
Mumford--Shah image segmentation was presented by Nesi \cite{Ne93}.
Cremers and Soatto \cite{cremers2002,cremers2003,CrSo05} gave an extension of the Mumford-Shah functional from intensity segmentation to motion based segmentation in terms of a probabilistic framework. 

Rathi \etal investigated active contours for joint segmentation and optical flow extraction~\cite{RaVaTaYe05}. 
Brox \etal \cite{BrBrWe05} presented a Chan-Vese  type model for piecewise smooth motion extraction. 
For given fixed image data the decomposition of image sequences into regions of homogeneous motion is encoded in a set of level set functions and the regularity of the motion fields in these distinct regions is controlled by a total variation functional. 
Indeed, Kornprobst \etal \cite{KoDeAu99} already studied a joint approach for the segmentation of moving objects in front of a still background and the computation of the motion velocities.

The paper is organized as follows. 
In Section \ref{sec: Relaxation} we introduce the variational approach via the definition of a suitable energy on space time intensity maps $u$ and motion fields $v$ and retrieve the general relaxation result for this type of energies.
The functional we propose to use is given 
in (\ref{eqfstar}-\ref{eq:definitionK}).
On the jump sets of $u$ and $v$ the integrand of the relaxed functional involves a microscopic variational problem.
The main contribution of this paper is to establish bounds for this microscopic energy and to give sufficient conditions for 
the non existence of microscopic oscillations.
In Section \ref{sec:consequences} we discuss the consequences of our results for optical flow estimation.  
Then, in Section \ref{sec:proofs} we present the proofs of  the  results discussed in  Section \ref{sec: Relaxation}.

\section{Variational approach and relaxation results} \label{sec: Relaxation}
In this section we  derive a joint functional for the restoration of a space time image sequence and the 
estimation of the underlying motion field. We start by fixing some notation.
We denote by $\Omega$ the image domain, a bounded Lipschitz-domain in $\R^d$ with $d\geq 1$, 
and by $D = (0,T) \times \Omega \subset \R^{d+1}$ the associated space time domain.
For a spatial vector $x \in \R^d$ we write $x = (x_1, \dots, x_d)$, while time-space vectors are denoted by
$y= (y_0, y_1, \ldots y_d)\in \R^{d+1}$ with $t=y_0$ being the time coordinate. 
Correspondingly,  the space time gradient reads as $\nabla=\nabla_{y} = (\partial_t, \nabla_x)$.
We use $| \cdot |$ for the Euclidean norm including all matrix spaces.
We use standard notation $L^p$  and $W^{1,p}$ with $1 \leq p \le \infty$ for Lebesgue and Sobolev spaces, respectively. 
By the fundamental decomposition result \cite{AmFuPa00,EvGa92} the derivative of a function of bounded variation 
$f \in BV\left(\R^k;\R^l\right)$  can be written as
\begin{equation*}
 Df = \nabla f \mathcal{L}^k + D^c f + [f]\otimes n \, \mathcal{H}^{k-1}\LL J_f\, 
\end{equation*}
where $\left(a \otimes b\right)_{i,j} = a_i b_j$ for $i=1, \cdots,l$ and $j=1, \cdots, k$.
Here $J_f$ is the jump set of $f$, $[f] = (f^+ - f^-)$ is the jump,  $f^+,f^-: J_f \to \R^l$ are the approximate limits on $J_f$, $n:J_f\to  S^{k-1}$ is the measure theoretic normal to $J_f$ and $D^c f$ is the Cantor part of the measure $Df$, orthogonal to both $\mathcal{L}^k$ and $\mathcal{H}^{k-1}\LL J_f$. 
The jump part of $Df$ is denoted by 
$D^jf=[f]\otimes n \, \mathcal{H}^{k-1}\LL J_f$.

Let $u_0  \in L^p(D)$ represent the input data, 
 a given grey valued image sequence,   with $p < 1^* := \frac{d+1}{d}$.
Our aim is to determine  a restored image sequence $u: D \to \R$ and an underlying motion field $v: D \to \R^d$.
Throughout this paper we use the shortcut notation $w=(1,v)$ for the space time motion field.
Since we expect that both the reconstructed  image sequence and the
reconstructed image velocities will jump on the boundaries of reconstructed
moving objects, which are codimension $1$ surfaces in  space time,
we consider $BV(D)$ as the suitable space for 
intensity maps and $BV(D;\R^d)$ as the suitable space for velocity fields.

We start by defining a functional $F$ measuring the quality of the restoration and the motion extraction for an
image sequence $u\in W^{1,1}(D)$ and a motion field $v\in W^{1,1}(D;\R^d)$. The actual functional on $BV(D)\times BV(D;\R^d)$ will then 
be defined via relaxation. For fixed $M$ and ${\alpha_F},\,{\alpha_v},\,{\alpha_u} >0$ we consider the following energy integrand
$g:\R^d\times \R^{d+1}\times\R^{d(d+1)}\to[0,\infty)$
\begin{equation*}
g\left(v,p, q\right) = {\alpha_F} \left| w \cdot p\right|  + {\alpha_v} \left| q \right| + {\alpha_u} \left|p\right|  
\end{equation*}
and define for a general Lipschitz domain $U \subset \R^{d+1}$ the energy
\beqn
E[u,v,U] = \begin{cases} 
\int_{U} g\left(v,\nabla u,\nabla v\right) \text{ d}y &\text{ if } (u,v) \in W^{1,1}(U) \times W^{1,1}(U;\R^d);\, \|v\|_{L^{\infty}} \leq M \\
\infty &\text{ otherwise.}
\end{cases}
\label{eq: energy}
\eeqn 
This energy is then complemented 
 with a fitting term with respect to the given image 
sequence $u_0\in L^p(D)$ to obtain the functional
\begin{equation*}
F[u,v,D] = \left\| u - u_0\right\|^p_{L^p\left(D\right)} + E[u,v,D]\,.
\end{equation*}
The energy $E[u,v,D]$ is finite for $(u,v) \in W^{1,1}(D) \times W^{1,1}(D;\R^d)$
and $ \|v\|_{L^{\infty}(D)} \leq M$
with ${\alpha_v} \int_D |\nabla v| \d y$ and ${\alpha_u} \int_D |\nabla u| \d y$ measuring the regularity of the motion field $v$ and the 
image sequence $u$, respectively. Furthermore, 
$${\alpha_F} \int_D |w\cdot \nabla u| \d y = {\alpha_F} \int_D |\partial_t u + v \cdot \nabla_x u| \d x \d t$$
quantifies the agreement of the pair $(u,\, v)$ with the brightness constancy constraint \eqref{eq:BCCE}.
We remark that for general input image sequences $u_0$ bounds on the energy do 
not imply a priori 
bounds on the motion field in $L^{\infty}$, 
for example in the case that $u_0$ is spatially uniform.
The constraint on $\|v\|_{L^{\infty}}$ has been included to avoid
 this technical difficulty.

At this point we are ready  to define the actual functional of interest for $(u,v) \in BV(D) \times BV(D;\R^d)$
as the relaxation $F^*$ of the functional $F$ with respect to convergence in $L^1(D)$:
\beqn
F^*[u,v,D] = \inf \left\{ \liminf_{k\to \infty} F[u_k,v_k,D] \,\Big|\, (u_k,v_k) \in L^1(D;\R^{d+1});\, (u_k,v_k) \!\stackrel{L^1}{\rightarrow} \!(u,v) \right\}\,.
\label{eq:Fstar}
\eeqn 
As already mentioned in the introduction in the generic case of optical flow applications the jump set $J_v$ of the motion field (the union of boundaries of moving objects)  is a subset of the jump set  $J_u$ of the image intensity (the union of all image edges).
Hence, $v$ is expected to jump on a subset of the support of the jump
 part $D^ju=[u]\otimes n \mathcal{H}^{d}\LL J_u$
of the measure $Du$, so that the term
$(1,v)\cdot Du$ is ill-defined. A proper understanding of this term requires
to select locally a microscopic profile 
for $u$ and $v$ and 
includes ---but will not be restricted to--- the proper choice 
between $v^+$ or $v^-$ for the pointwise value for $v$
in the term $(1,v)\cdot D^j u$ and thus the local selection between foreground and background (cf. the consistency issue (ii) in the simple model problem in Section \ref{sec:intro}).
In fact, the theory of relaxation  for problems with linear growth is more complex than
the one with $p$-growth, $p>1$  \cite{Da89}, because of the singular part of the gradient  in the limit.
Relaxation and lower semicontinuity with a convex integrand depending only on the gradient field were already obtained in
the 60s \cite{GoffmanSerrin1964,Resetnjak1967}, the general case, with
dependence of the integrand also on $x$ and $u$ was investigated in
 the early 90s
\cite{AmbrosioDalmaso1992,AvilesGiga1991}. Here, we use the more general
 result by Fonseca and M\"uller in \cite{FoM93}, which also includes the
case  of quasiconvex integrands,  as the starting point of our investigation.
Successive developments include \cite{BoFoMa98,BoFoLeMa02,FuscoGoriMaggi2006,KristensenRindler2010,Rindler2012}.

These results show that  the  relaxation $F^*$ of $F$ 
with respect to the $L^1$-topology, as defined in (\ref{eq:Fstar}),
is finite on 
the domain
\begin{equation*} 
BV(D) \times \left(BV(D; \R^{d}) \cap \left\{v \in L^{\infty}: \left\| v \right\|_{L^{\infty}} \leq M \right\}\right)
\end{equation*}
and it equals
\begin{align}\nonumber
F^*\left[u,v,D\right] = &\left\| u - u_0 \right\|_{L^p\left(D\right)}^p 
+ \int_{D} g\left(v,\nabla u, \nabla v\right) \text{ d}y
+ \int_{D} g\left(v,D^c u, D^c v\right) 
\\+ &\int_{J_{(u,v)}} K(u^+,u^-,v^+,v^-,n) \text{ d}\mathcal{H}^{d}.
\label{eqfstar}
\end{align}
Since $g$ is one-homogeneous and convex in the second and third arguments, 
$g$ coincides with its regression function and
the third term in (\ref{eqfstar}) should be interpreted as 
\begin{equation*}
  \int_D g\left(v, \frac{d D^cu}{d(|D^cu|+|D^cv|)}, \frac{d D^cv}{d(|D^cu|+|D^cv|)}\right)
d(|D^cu|+|D^cv|)\,.
\end{equation*}
The key nonconvexity arises from the dependence of $g$ on the first variable
in the jump part.
Thereby, the function $K: \mathbb{R} \times \mathbb{R} \times \mathbb{R}^{d} \times \mathbb{R}^{d} \times S^d \rightarrow \mathbb{R}$ 
depending on the approximate limits $\left(u^{+},v^+\right)$ and $\left(u^{-},v^-\right)$ of $(u,v)$ 
and the measure theoretic normal $n$ on the jump set $J_{(u,v)}$ is the solution of a local minimizing problem in which the energy on the jump is optimized with respect to all possible microstructures. Precisely, 
\begin{equation}\label{eq:definitionK} 
 K\left(u^{+}, u^{-},v^{+},v^{-}, n\right) = \inf \left\{ E[u,v,Q_{n}]: (u,v) \in \mathcal{A} \right\}\,,
\end{equation}
where $Q_n$ is the rotated cube
\begin{equation*}
Q_{n} = \left\{ y \in \mathbb{R}^{d+1}: \left| y \cdot n \right| < \frac{1}{2}, \left| y \cdot m^1\right|  < \frac{1}{2}, \dots, \left| y \cdot m^{d} \right| < \frac{1}{2} \right\},
\end{equation*}
with $\{n,m^1,\dots,m^{d}\}$ denoting an orthonormal basis of $\mathbb{R}^{d+1}$,
and  $\mathcal{A}$ is the set of $W^{1,1}$-functions which have traces
$u^{\pm}, v^{\pm}$ on the two sides of $Q_{n}$ normal to $n$ and are periodic
in the $m^1, \dots, m^{d}$- directions,
\begin{align*}
\mathcal{A} = \Bigg\{& (u,v) \in W^{1,1}\left(Q_n;\R^{1+d}\right): u = u^{\pm} \text{ and } v = v^{\pm} \text{ on } \partial Q_{n} \cap \left\{ y \cdot n = \pm \frac{1}{2} \right\}, \\ 
&u(y) = u\left(y + m^i\right) \text{ and }  v(y) = v\left(y+ m^i\right) \text{ on } \left\{y\cdot m^i = - \frac12\right\} \Bigg\}.
\end{align*}

We remark that the term $\left\| u - u_0 \right\|_{L^p\left(D\right)}^p $ is continuous in $u$ with respect to the weak convergence in $BV$, since we chose $p<1^*$.
Furthermore, both $F$ and $F^*$ are coercive in $BV$, in the sense that for any given $u_0$ and for any  sequence $(u_j,v_j)$ with $F^*[u_j,v_j,D]$  
bounded, the sequence of $BV$ norms of $u_j$ and $v_j$ are also bounded. 
Therefore the above representation of the relaxation 
follows immediately from the  more general statement of Fonseca and
M\"uller \cite[Theorem 2.16]{FoM93} and existence of minimizers for the
 relaxed functional $F^*$ follows easily by
the direct method of the calculus of variations.

A practical usage of the functional $F^*$ requires knowledge of the effective 
surface energy $K$, much as in the case of relaxation on $W^{1,p}$ spaces 
one needs to determine the quasiconvex envelope of the integrand
\cite{Da89}. A numerical computation is in principle feasible 
via the minimization in (\ref{eq:definitionK}), but it is nevertheless useful 
to extract analytical information on $K$ as far as possible.
In what follows we give lower and upper bounds for the singular term $K$ 
and compute it explicitly in special cases. 
Thereby, we show that the model favors locally simple (i.e.~planar) profiles in the microscopic problem \eqref{eq:definitionK} 
under reasonable assumptions from the viewpoint of the practical optical flow application. 
This renders the functional $F^*$ well-suited for joint image sequence restoration and motion extraction. 
A detailed discussion of the consequences of the bounds for $K$ is postponed to Section \ref{sec:consequences}.

\begin{thrm}\label{theorem:bounds}
Given $u^+,u^- \in \R$, $v^+,v^- \in \R^d$ and $n\in S^d$ for $K = K\left(u^+,u^-,v^+,v^-,n\right)$  the following statements hold:
\begin{enumerate}
\item  \label{item:ofcsatisfied} 
If $\left(w^+ \cdot n\right) \left(w^- \cdot n\right)\le 0$ then   
$K = {\alpha_u} \left|[u]\right| + {\alpha_v} \left|[v]\right|\,,$ 
\item \label{item:1dbounds}  
For $d=1$ and $\left|[u]\right|  \leq \frac{2{\alpha_v}}{{\alpha_F}}$ one has 
 $\min\limits_{N^+ \in\R^2}\underline{K}(N^+) \leq K \leq \min\limits_{N^+ \in\R^2} \overbar{K}(N^+)$, where
\begin{align*}
 \overbar{K}(N^+) &= 
\left({\alpha_u} \left|[u]\right| + {\alpha_v} \left|[v]\right|\right)\left(\left|N^+\right| + \left|N^-\right|\right) 
 + {\alpha_F} \left|[u]\right| \left( \left|N^+ \cdot w^+\right| + \left| N^- \cdot w^-\right| \right)\,, \\
 \underline{K}(N^+) &= 
\left({\alpha_u}  \left|[u]\right| + 
{\alpha_v} \left|[v]\right|\right)\left(\left|N^+\right| + \left|N^-\right|\right) 
 + {\alpha_F} \left|[u]\right|  \left|N^+ \cdot w^+ + N^- \cdot w^-\right|\,,
\end{align*}
with $N^- = n - N^+$,
\item \label{item:ddbounds} 
For general  $d\ge 1$ one has  $K\le \overbar{K}_d$, where
\begin{equation*}
 \overbar{K}_d = \min \!\left(\sum_{j = 1}^l \left( {\alpha_u} \left|[u]\right| +
   {\alpha_v} \left|v^+\! - \!v^j\right|+  {\alpha_v} \left|v^-\! -\! v^j\right|\right)
 \left|N^j\right|  
 + {\alpha_F} \left|[u]\right| \sum_{j = 1}^l \left|N^{j}\! \cdot \!w^j\right|\right)\,,
\end{equation*}
where $w^j = \left(1,v^j\right)$ and the minimum is taken over $l\in\N$ and the
set of vectors  $v^1, \dots, v^l \in \R^d$  and $N^1, \dots N^l \in \R^{d+1}$,
subject to  $\sum_{j=1}^l N^j = n$.
 \end{enumerate}
\end{thrm}
The proof is given in  Section \ref{sec:thmproofs}.

The upper bounds in  Theorem \ref{theorem:bounds} are based on suitable choices for the optimal microscopic solution $u$, $v$
of \eqref{eq:definitionK}. If the profile of these microscopic solutions $u,v$ depend only on $y\cdot n$, we call it \emph{simple}.
There are in particular two types of simple profiles when solving \eqref{eq:definitionK} for given data $(u^-,v^-)$, $(u^+,v^+)$, and $n$.
In the case $\left(w^+ \cdot n\right)\left(w^- \cdot n\right)\le 0$
 we define the following piecewise constant
functions $u\in BV(Q_n)$ and $v\in BV(Q_n;\R^d)$:
\begin{equation}\label{eqconstrvvva}
 u(y) = \begin{cases}
         u^- &\text{ if } y \cdot n < 0 \\
         u^+ &\text{ if } y \cdot n \geq 0
        \end{cases}
\hskip3mm\text{ and } \hskip3mm
v(y) = \begin{cases}
        v^- &\text{ if } y \cdot n < -\frac13 \\
        v^0 &\text{ if } -\frac13 \leq y \cdot n < \frac13 \\
        v^+ &\text{ if } y \cdot n \geq \frac13
       \end{cases}\,,  
\end{equation}
where $v^0$ is chosen in the line segment $[v^-,v^+]$ such that 
$w^0 \cdot n =0$ for $w^0=(1,v^0)$.
For suitable approximations in $W^{1,1}$ such as those obtained by convolution $u_k = \frac{k}{2}\chi_{\{|y\cdot n| <\frac{1}{k}\}} \ast u$, 
$v_k = \frac{k}{2}\chi_{\{|y\cdot n| <\frac{1}{k}\}} \ast v$ we have that 
\beqn\label{eq:Kopt1}
K(u^+,u^-,v^+,v^-,n) \leq \liminf_{k\to \infty} E[u_k,v_k,Q_k] = {\alpha_v} |[v]| + {\alpha_u} |[u]|\,,
\eeqn
where $[v] = v^+-v^-$ and $[u] = u^+-u^-$. 
This profile corresponds to a  microscopic consistency with a $BV$ interpretation of the brightness constancy constraint
\eqref{eq:BCCE}.
In the other case $\left(w^+ \cdot n\right) \left(w^- \cdot n\right)>0$ 
consistency of a simple and optimal microscopic profile is out of reach. Let us assume that
 $\left|w^{-} \cdot n\right| \leq \left|w^{+} \cdot n\right|$.
Then a feasible, simple profile for the minimization of $E[\cdot, \cdot, Q_n]$ is given by
\beq
  u(y) = \begin{cases}
         u^- &\text{ if } y \cdot n < 0 \\
         u^+ &\text{ if } y \cdot n \geq 0
        \end{cases}
\hskip3mm \text{ and } \hskip3mm
  v(y) = \begin{cases}
         v^- &\text{ if } y \cdot n < \frac{1}3 \\
         v^+ &\text{ if } y \cdot n \geq  \frac{1}3
        \end{cases}\,.
\eeq
Indeed, again using the same ansatz for the approximation as above, one obtains
\beqn\label{eq:Kopt2}
K(u^+,u^-,v^+,v^-,n) \leq \liminf_{k\to \infty} E[u_k,v_k,Q_k] = {\alpha_F} |w^{-}\cdot n| + {\alpha_v} |[v]| + {\alpha_u} |[u]|\,.
\eeqn
Figure \ref{figure: simpleprofiles} shows sketches of the simple profiles in both cases.
In the next theorem we establish sufficient conditions for the existence of minimizing profiles which are simple.
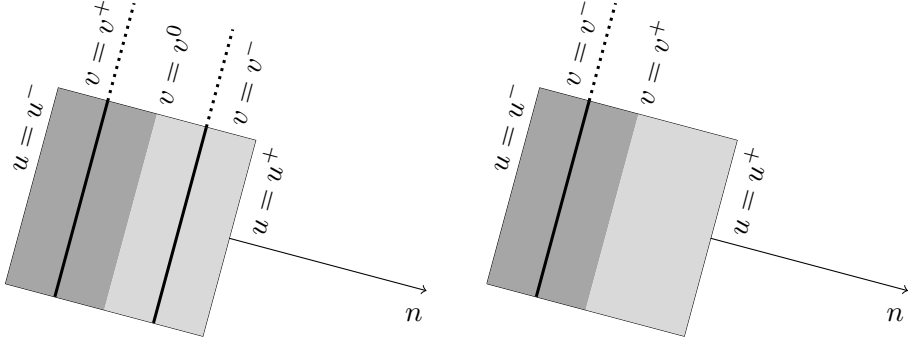
\begin{figure}
 \centering
\begin{tikzpicture}[rotate=300, scale=0.95]
\draw (0,0) -- (-2,2) -- (0,4) -- (2,2) -- (0,0);
\fill[gray!70] (1,1) -- (-1,3) -- (-2,2) -- (0,0) -- (1,1);
\fill[gray!30] (1,1) -- (-1,3) -- (0,4) -- (2,2) -- (1,1);
\draw[black,line width=0.04cm] (0.5,0.5) -- (-1.5,2.5);
\draw[black,line width=0.04cm](1.5,1.5) -- (-0.5,3.5);
\draw[black,dotted,line width=0.04cm] (-1.5,2.5) -- (-2.5,3.5);
\draw[black,dotted,line width=0.04cm] (-0.5,3.5) -- (-1.5,4.5);
\draw(-1,3) node[black, anchor= west, rotate=75]{$v=v^0$};
\draw[->] (1,3) -- (3,5);
\draw(3,4.8) node[black, anchor=north]{$n$};
\draw(-1.75,2.25) node [black, anchor= west, rotate=75]{$v=v^+$};
\draw(-1.55,1.55) node[black, anchor= south, rotate=75]{$u=u^-$};
\draw(-0.25,3.75) node [black, anchor= west, rotate=75]{$v=v^-$};
\draw(0.45,3.55) node[black, anchor= north, rotate=75]{$u=u^+$};
\end{tikzpicture}
\quad
\begin{tikzpicture}[rotate=300, scale=0.95]
\draw (0,0) -- (-2,2) -- (0,4) -- (2,2) -- (0,0);
\fill[gray!70] (1,1) -- (-1,3) -- (-2,2) -- (0,0) -- (1,1);
\fill[gray!30] (1,1) -- (-1,3) -- (0,4) -- (2,2) -- (1,1);
\draw[black,line width=0.04cm] (0.5,0.5) -- (-1.5,2.5);
\draw[black,dotted,line width=0.04cm] (-1.5,2.5) -- (-2.5,3.5);
\draw[->] (1,3) -- (3,5);
\draw(3,4.8) node[black, anchor=north]{$n$};
\draw(-1,3) node [black, anchor=west, rotate=75]{$v=v^+$};
\draw(-1.55,1.55) node[black, anchor= south, rotate=75]{$u=u^-$};
\draw(-1.75,2.25) node[black, anchor=west,rotate=75]{$v=v^-$};
\draw(0.45,3.55) node[black, anchor= north, rotate =75]{$u=u^+$};
\end{tikzpicture}
\caption{A sketch of the two simple profiles for $\left(w^+ \cdot n\right) \left(w^- \cdot n\right) \leq 0$ (left) and $\left(w^+ \cdot n\right) \left(w^- \cdot n\right)>0$ with $\left|w^{-} \cdot n\right| \leq \left|w^{+} \cdot n\right|$ (right).
The black line indicates the jump set of $v$. In the dark grey area $u$ takes the value $u^-$, on the light grey area the value $u^+$.}
\label{figure: simpleprofiles}
\end{figure}
\begin{thrm}\label{theorem:sufficient}
Let $u^+,u^- \in \R$, $v^+,v^- \in \R^d$ and $n\in S^d$, then we have the following.
\begin{enumerate}
 \item
\label{theorem:sufficientit1} If $\left(w^+ \cdot n\right) \left(w^- \cdot n\right)\le 0$ 
then the optimal profile for $K$ is simple and $$K={\alpha_v} |[v]| + {\alpha_u} |[u]|\,.$$
 \item
 \label{item:theoremsufficient1d}
If $d=1$, $(w^-\cdot n)(w^+\cdot n) > 0$, $|[u]|\le \frac{2{\alpha_v}}{{\alpha_F}}$ and 
$2\left({\alpha_u} |[u]| + {\alpha_v}  |[v]|\right) |n_1| \geq {\alpha_F} |[u]|  |[v]|$
then the minimizing profile for $K$ is simple and 
\begin{equation}\label{eq:Ksimple}
K= {\alpha_F} 
\min\{|w^-\cdot n|, |w^+\cdot n|\}  |[u]| + {\alpha_v} |[v]| + {\alpha_u} |[u]|\,.
\end{equation}
\item If $d>1$, $(w^-\cdot n)(w^+\cdot n) > 0$ and $|[u]|  \le \frac{2{\alpha_v}}{d {\alpha_F}} \, \frac{ |[w] \cdot n| }{|[v]|}$
then the minimizing profile for $K$ is simple and given by \eqref{eq:Ksimple}\,.
\label{item:theoremsufficient}
\item
 \label{item:simpletok}
If the profile is simple, then 
\begin{equation*}
  K= \min_{a\in\R^d} 
\left( {\alpha_u} \left|[u]\right| +
   {\alpha_v} \left|v^+\! - \!a\right|+  {\alpha_v} \left|v^-\! -\! a\right|
 + {\alpha_F} \left|[u]\right|  \left|n \cdot (1,a)\right|\right)\,.
\end{equation*}
\end{enumerate}
\end{thrm}
The proof is given in  Section \ref{sec:thmproofs}. An example where the profile is not simple is discussed 
in Section \ref{sec:consequences} below.

\section{Consequences for the motion estimation}\label{sec:consequences} 
In this section  we  discuss the practical implications of Theorem \ref{theorem:sufficient}  
on the actual optical flow estimation based on the proposed variational approach for the joint image sequence restoration and motion extraction. 
We will analyze implicit conditions on the approximate limits of $u$ and $v$ at jump sets under which the singular energy density $K$ defined in 
\eqref{eq:definitionK} is simple and there is no relevant microscopic scale arising in the variational model.  
Furthermore, we will give an example where microstructures actually appear and thus the singular energy density $K$ is not simple.
As the only interesting cases appear where $[u] \neq 0$, in this section we always assume that $J_v \subset J_u$.

We start observing  that the condition for simple profiles in  one space dimension in Theorem \ref{theorem:sufficient} (ii) includes in 
particular the case
\beqn\label{eq:simplecondition}
\left|[u]\right| \leq 2 \frac{{\alpha_v}}{{\alpha_F}} \left|n_1\right|\,,
\eeqn
as well as  the case
\beqn\label{eq:simplecondition2}
 \left|[v]\right| \leq 2 \frac{{\alpha_u}}{{\alpha_F}} \left|n_1\right| \text{ and } \left|[u]\right| \leq 2 \frac{{\alpha_v}}{{\alpha_F}}\,.
\eeqn
Hence, for moderate speed of the intensity interface $J_u$  (associated with large $n_1$) 
and either moderate difference $|[v]|$
of the estimated motion on both sides of the interface or moderate intensity variation $|[u]|$ the condition
in Theorem \ref{theorem:sufficient} (ii) is fulfilled
(in particular for ($|n_1| > \frac12$ and $\frac{{\alpha_v}}{{2\alpha_F}} \geq \|u\|_{L^\infty}$) or ($\frac{{\alpha_v}}{{\alpha_F}} \geq \|u\|_{L^\infty}$ and $|n_1| \geq \frac{M \alpha_F}{{\alpha_u}}$)). For the case $d>1$ the condition in Theorem \ref{theorem:sufficient} (iii) can be rephrased as 
$$
|[v] \cdot (n_1,\ldots, n_d)| \geq \frac{d {\alpha_F}}{2{\alpha_v}} |[v]|\, |[u]|\,.
$$
Thus, 
if  $\frac{{\alpha_F}}{{\alpha_v}}$ is small,
this condition is fulfilled
 for a moderate speed of the intensity interface $J_u$ (associated with large spatial component of $n$) and for 
a direction of the jump $\frac{[v]}{|[v]|}$ with a significant component 
 pointing in direction of the spatial interface normal 
$\frac{(n_1,\ldots, n_d)}{\|(n_1,\ldots, n_d)\|}$. 

In the practical application the velocities on the two sides of the discontinuity  $v^+$ and $v^-$ are determined within the variational  setting
by shading or texture information on both sides of
the edge set $J_u$. Then, the singular energy density $K(u^{+}, u^{-},v^{+},v^{-}, n)$ is associated with the proper identification of the type of object motion as described for the simple model problem in Section \ref{sec:intro}, 
in particular the decision on foreground or background and the identification of additional 
erosion or dilation. As we will see below, for $K$ simple 
the minimization of the joint functional is able to decide on the motion pattern. 
In the case that $K$ is not simple and that microstructures appear the 
variational model seems not to appropriately reflect the scope of possible motion patterns. 
On the other hand, under the reasonable implicit assumption that the data $u^{+},\,  u^{-},\, v^{+},\, v^{-}$ fulfills one of the 
conditions under which the singular energy density $K$ is simple the joint variational approach actually renders the 
coupled restoration and motion estimation problem meaningful including the proper identification of the local motion pattern at object edges.

In what follows we study the different local motion pattern at some point $y \in J_u$ and the associated
singular energy density $K$ in more detail. 
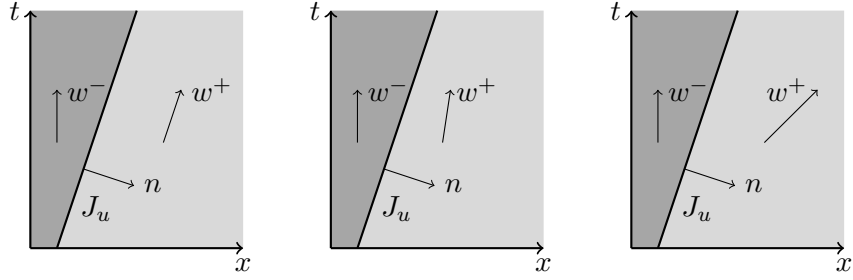
\begin{figure}
 \centering
\begin{tikzpicture}[scale = 0.7]

\fill[gray!70] (0,0) -- (0,4.5) -- (2,4.5) -- (0.5,0) -- (0,0);
\fill[gray!30] (0.5,0) -- (2,4.5) -- (4,4.5) -- (4,0) -- (0.5,0);

\draw[thick,->] (0,0) -- (4,0);
\draw[thick,->] (0,0) -- (0,4.5);

\draw (4,0) node[anchor=north] {$x$};
\draw (0,4.5) node[anchor=east] {$t$};

\draw[thick] (0.5,0) -- (2,4.5);
\draw[->] (1,1.5) -- (1+3/sqrt{10}, 1.5 - 1/sqrt{10});

\draw (1+3/sqrt{10}, 1.5 - 1/sqrt{10}) node[anchor=west] {$n$};
\draw (0.75,0.75) node[anchor=west] {$J_u$};

\draw[->] (0.5,2) -- (0.5,3);
\draw[->] (2.5,2) -- (2.5 + 1/3,3);

\draw (0.5,3) node[anchor = west] {$w^-$};
\draw (2.5+1/3,3) node[anchor = west] {$w^+$};
\end{tikzpicture}
\quad
\begin{tikzpicture}[scale = 0.7]

\fill[gray!70] (0,0) -- (0,4.5) -- (2,4.5) -- (0.5,0) -- (0,0);
\fill[gray!30] (0.5,0) -- (2,4.5) -- (4,4.5) -- (4,0) -- (0.5,0);

\draw[thick,->] (0,0) -- (4,0);
\draw[thick,->] (0,0) -- (0,4.5);

\draw (4,0) node[anchor=north] {$x$};
\draw (0,4.5) node[anchor=east] {$t$};

\draw[thick] (0.5,0) -- (2,4.5);
\draw[->] (1,1.5) -- (1+3/sqrt{10}, 1.5 - 1/sqrt{10});

\draw (1+3/sqrt{10}, 1.5 - 1/sqrt{10}) node[anchor=west] {$n$};
\draw (0.75,0.75) node[anchor=west] {$J_u$};

\draw[->] (0.5,2) -- (0.5,3);
\draw[->] (2.1,2) -- (2.25,3);

\draw (0.5,3) node[anchor = west] {$w^-$};
\draw (2.5-1/3,3) node[anchor = west] {$w^+$};
\end{tikzpicture}
\quad
\begin{tikzpicture}[scale = 0.7]

\fill[gray!70] (0,0) -- (0,4.5) -- (2,4.5) -- (0.5,0) -- (0,0);
\fill[gray!30] (0.5,0) -- (2,4.5) -- (4,4.5) -- (4,0) -- (0.5,0);

\draw[thick,->] (0,0) -- (4,0);
\draw[thick,->] (0,0) -- (0,4.5);

\draw (4,0) node[anchor=north] {$x$};
\draw (0,4.5) node[anchor=east] {$t$};

\draw[thick] (0.5,0) -- (2,4.5);
\draw[->] (1,1.5) -- (1+3/sqrt{10}, 1.5 - 1/sqrt{10});

\draw (1+3/sqrt{10}, 1.5 - 1/sqrt{10}) node[anchor=west] {$n$};
\draw (0.75,0.75) node[anchor=west] {$J_u$};

\draw[->] (0.5,2) -- (0.5,3);
\draw[->] (2.5,2) -- (2.5 + 1,3);

\draw (0.5,3) node[anchor = west] {$w^-$};
\draw (2.5+1,3) node[anchor = east] {$w^+$};
\end{tikzpicture}
\caption{1D Sketch of different configurations at the interface $J_u$ of a moving object $(u=u^+)$ with motion velocity $v^+$ on a background with motion field $v^-$ and $w^{\pm} = \left(1,v^{\pm}\right)$. 
On the left the interface motion is consistent with the object motion, in the middle the object is eroding and on the right the object is dilating.}
\label{fig:cases}
\end{figure}

\subsubsection*{Consistent interface motion.} Let us suppose that a light object with image intensity $u^+$ is moving with a  speed $v^+$ on a dark background with speed $v^-$ and image intensity $u^-$ ($u^+ > u^-$) (cf. left sketch in Fig. \ref{fig:cases}). 
The consistency of the interface motion of the space time edge set $J_u$ with the object motion 
is expressed in terms of the brightness constancy assumption $w^+ \cdot n =0$ at the point $y\in J_u$ with $n$ denoting the space time normal on $J_u$,
i.e.~$D u = [u] \otimes n \, \mathcal{H}^{d}\LL J_u$. In this case Theorem \ref{theorem:sufficient} (i)
applies and $K={\alpha_v} |[v]| + {\alpha_u} |[u]|$. Microscopically this is realized by the profile sketched in Fig. \ref{figure: simpleprofiles} (left) with $v^0 = v^+$. Due to noise the brightness constancy assumptions might only be approximately fulfilled with $|w^+ \cdot n| \ll 1$. Then, Theorem \ref{theorem:sufficient} (ii) shows that we obtain the corresponding approximate singular energy density $K= {\alpha_F} |w^+ \cdot n| |[u]| + {\alpha_v} |[v]| + {\alpha_u} |[u]|$.

\subsubsection*{Non consistent interface motion.} Let us suppose that in the same configuration  $w^-\cdot n<w^+\cdot n<0$ (cf. middle sketch in Fig. \ref{fig:cases}). Hence, neither the object motion $v^+$ nor the motion $v^-$ (currently classified as background) is consistent with the motion of the interface $J_u$. Indeed, we observe an erosion of the interface. Let us assume that $v$ is the actual speed of the interface
with $w\cdot n=0$ for $w=(1,v)$, then $v-v^+$ is the effective erosion velocity and 
$|(0,v-v^+)\cdot n| = |w^+\cdot n|$ is the associated footprint in the singular energy density $K$, which is in agreement with the findings of Theorem \ref{theorem:sufficient} (ii). Obviously, the energy functional considers this classification as favorable compared to the one obtained via flipping object and background and classifying 
a foreground object with intensity $u^-$ moving with speed $v^-$ and dilation speed $v^- -v$ and larger footprint in the singular energy density $|w^-\cdot n|$.
Microscopically this is realized by the profile sketched in Fig. \ref{figure: simpleprofiles} (right). 
If  $w^-\cdot n < 0 <  w^+\cdot n$
(cf.~right sketch in Fig. \ref{fig:cases}) the variational approach favors the classification of a dilation process on the interface $J_u$ with velocity $\bar v -v^+$
for the velocity $\bar v$ on the line segment $[v^-,v^+]$ with $\bar w \cdot n = (1,\bar v) \cdot n =0$ and thus smallest 
possible singular energy density  $K= {\alpha_v} |[v]| + {\alpha_u} |[u]|$. The associated microscopic solution profile coincides with that sketched in  Fig.~\ref{figure: simpleprofiles} (left).

\subsubsection*{Objects and background with constant intensity.}
\begin{wrapfigure}{r}{0.5\linewidth}
\begin{tikzpicture}[scale= 1.7, font=\tiny]
\fill[gray!10]

(0,0) -- (0,3) -- (3,3) -- (3,0) -- (0,0); 

\fill[gray!50] (1.5,1.75) circle(0.6);
\path[draw=red] (1.5,1.75) circle[radius=0.6];
\draw[red, ->] (1.5,1.75) -- (1.8,2.05);
\draw (1.65,1.9) node[anchor = south, color=red] {$v_1$};
\draw (1.2,1.9) node{$O_1$};

\fill[gray!30] (1.75,0.8) circle (0.5);
\path[draw=green!50!black] (1.75,0.8) circle[radius=0.5];
\draw[green!50!black, ->] (1.75,0.8) -- (2.2,0.65);
\draw (1.975,0.725) node[anchor = south, color=green!50!black] {$v_2$};
\draw (1.75, 1.1) node{$O_2$};

\fill[gray!70] (1,1) circle (0.5);
\path[draw=blue] (1,1) circle[radius=0.5];
\draw[blue, ->] (1,1) -- (1,1.4);
\draw (1,1.2) node[anchor = east, color=blue] {$v_3$};
\draw (1.25,0.8) node{$O_3$};

\end{tikzpicture} 
\vspace{-10pt}
\caption{Overlaying motion.}
\label{figure:multipleobjects}
\end{wrapfigure}
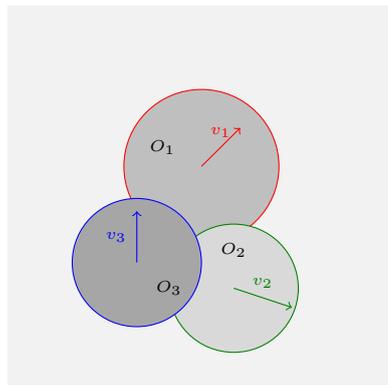
Objects and background with constant intensity.
If there is no shading or texture information, the estimation of motion velocities can solely be based on the observed 
motion of interfaces. Examplarily, let us suppose that objects $\object_1, \ldots, \object_m$ are moving with constant velocities $v_1, \ldots, v_m$ in front of a immobile background as sketched in Figure \ref{figure:multipleobjects}. 
Furthermore, let us explicitly rule out interface dilation and erosion.
For such a configuration the singular energy density compares the different foreground and background configurations
in the local depth ordering. 
For a fixed depth ordering of the objects the estimation of a velocity $v_i$ for each object is based on the minimization of $\int_{\partial \object_{i}^{vis}} K \d \mathcal{H}^d$ with 
\begin{equation*}
K= {\alpha_F} |(1,v_i)\cdot n| + {\alpha_v} |v_i - v_i^{opp}| +  {\alpha_u} |u_i - u_i^{opp}| \,, 
\end{equation*}
where ${\partial \object_{i}^{vis}}$ is the visible part of $\partial \object_{i}$ and  $v_i^{opp}$, $u_i^{opp}$
are the velocity and the intensity opposite $\partial \object_{i}$, respectively. 
The estimation of the depth ordering via the variational approach is then performed by a comparison of the minimal total energy obtained for the set 
of all possible depth configurations (In Figure \ref{figure:multipleobjects} the color of the interfaces shows which is the relevant velocity of the singular energy density for the minimal total energy).
An estimation of the background speed is obviously not possible.

\subsubsection*{Appearance of microstructures.} 
By the above discussion 
 it is clear that we have to deal with 
large relative velocities and jumps in the intensities.  The key idea is to find a case where the 
upper bound in   Theorem \ref{theorem:bounds}\ref{item:1dbounds} 
and \ref{item:ddbounds} 
is lower than the energy of simple profiles given in Theorem \ref{theorem:sufficient}\ref{item:simpletok}.

We discuss first the case $d=1$. Let $n \in S^1$ and let $\overbar{K}$ be the upper bound given in  Theorem \ref{theorem:bounds}\ref{item:1dbounds}. We assume that
\begin{equation*}
|[u]|\le \frac{2{\alpha_v}}{{\alpha_F}} \hskip3mm\text{ and }\hskip3mm
0< w^-\cdot n <w^+\cdot n\,.  
\end{equation*}
Then the energy of a simple profile in  Theorem \ref{theorem:sufficient}\ref{item:simpletok}
coincides with $\overline K(0)$.
Indeed, if $a$ is between $v^-$ and $v^+$ then $|v^+-a|+|v^--a|\ge |v^+-v^-|$ and $|n\cdot (1,a)|\ge |n\cdot w^-|$; if $a$ is outside that interval then,
letting $b$ be the projection of $a$ on the interval,
$|v^+-a|+|v^--a|=|v^+-v^-|+2|a-b|$ and $|n\cdot (1,a)|\ge |n\cdot b|-|b-a|$.
This proves that the minimum is attained at $a=v^-$.

Therefore it suffices to construct a
situation where $0$ is not a minimizer of the upper bound function $\overline K$. For a vector $N\in\R^2$ chosen later and a small $\delta\in (0,1)$ we compute,
assuming $n\cdot w^-\ne 0$ and writing for brevity
 $\zeta = \frac{{\alpha_F} \left|[u]\right|}{{\alpha_u} \left|[u]\right| + {\alpha_v}\left|[v]\right|}$,
\begin{alignat*}1
  \frac{\overline K(\delta N)-\overline K(0)}
{\alpha_u \left|[u]\right| + {\alpha_v}\left|[v]\right|}
&=|\delta N| + |n-\delta N|-1
 + \zeta\,  \left( |\delta N \cdot w^+| + | (n-\delta N) \cdot w^-|
-| n \cdot w^-|
 \right) \\
&  =A(N) \delta  + O(\delta^2)
\end{alignat*}
where
\begin{equation*}
  A(N)=|N|-n\cdot N + \zeta (|N\cdot w^+|-N\cdot w^-).
\end{equation*}
Therefore it suffices to show that $N\in \R^2$ can be chosen so that $A(N)<0$.

We choose a large velocity $v\ge 1$, set 
 ${\alpha_F}={\alpha_v}={\alpha_u}=1$, $v^\pm=\pm v$, $u^+=1, u^- = 0$,
so that $[u]=1 \leq 2\alpha_v/\alpha_F$, $\zeta=1/(1+2v)$, $w^\pm=(1,\pm v)$.  
We choose 
a normal which corresponds to motion with a velocity
close to but not identical with the velocity $v^-$. 
Precisely, for some $\tilde v>v$ 
we set $n=(\tilde v,1)/\sqrt{\tilde v^2+1}$.
Then it is easy to see that $0<n\cdot w^-<n\cdot w^+$.
Finally  we set $N=(v,-1)$, so that $N\cdot w^+=0$. 
It remains to show that $v$ and $\tilde v$ can be chosen so that $A(N)<0$. To do this we compute
\begin{equation*}
A(N)=  |N|-n\cdot N + \zeta (|N\cdot w^+|-N\cdot w^-)
  =\sqrt{v^2+1} -\frac{v\tilde v-1}{\sqrt{\tilde v^2+1}} - \frac{2v}{1+2v}
  \end{equation*}
which is negative if $v$ and $\tilde v$ are chosen sufficiently large (cf. Fig. \ref{figure:slidingmotion}).
A detailed computation shows that $v=2$ and any $\tilde v \geq 4$ will do.
In the case $v=\tilde v$, however, the Taylor series above is not admissible (since $w^-\cdot n=0$) and the profile becomes simple again, in agreement with Theorem 
\ref{theorem:sufficient}\ref{theorem:sufficientit1}.

 \begin{figure}[t]
\centering

\begin{tikzpicture}[scale = 0.7]

\fill[gray!70] (0,0) -- (0,2.5) -- (4,1.5) -- (4,0) -- (0,0);
\fill[gray!30] (0,2.5) -- (0,4.5) -- (4,4.5) -- (4,1.5) -- (0,2.5);

\draw[thick,->] (0,0) -- (4,0);
\draw[thick,->] (0,0) -- (0,4.5);

\draw (4,0) node[anchor=north] {$x$};
\draw (0,4.5) node[anchor=east] {$t$};

\draw[thick] (0,2.5) -- (4,1.5);

\draw (2.4,1.9) -- (2.5,2.3) -- (1.7,2.5) -- (1.5, 1.7) -- (2.3, 1.5) -- (2.4,1.9);

\draw[->] (2.1,2.4) -- (2.3,3.2);

\draw (2.3,3.2) node[anchor=west] {$n$};

\end{tikzpicture}
\quad
\quad
\begin{tikzpicture}[scale = 0.7]
\draw (0,0) -- (1,4) -- (-3,5) -- (-4,1) -- (0,0);
\draw[->] (-1,4.5) -- (0,8.5);

\draw (0,8.5) node[anchor=west] {$n$};

\draw[thick]  (-3.5,3) -- (-0.5, 3/2) -- (0.5,2);

\fill[gray!70] (0,0) -- (-4,1) -- (-3.5,3) -- (-0.5, 3/2) -- (0.5,2) -- (0,0);
\fill[gray!30] (0.5,2) -- (1,4) -- (-3,5) -- (-3.5,3) -- (-0.5, 3/2) -- (0.5,2);

\draw[->] (-3/2,5/2) -- (-6/2,13/4);
\draw[->] (-3/2,5/2) -- (0,13/4);

\draw (-6/2,13/4) node[font=\small, anchor=south west] {$\frac14 w^-$};
\draw (0,13/4) node[font=\small, anchor=south] {$\frac14 w^+$};
\end{tikzpicture}
 \caption{Sketch of a situation where microstructures develop. On the left the 1D sketch of a moving object with relatively high velocity in space-time.
 On the right a microstructure profile that has a lower energy than a simple profile.}
 \label{figure:slidingmotion} 
\end{figure}
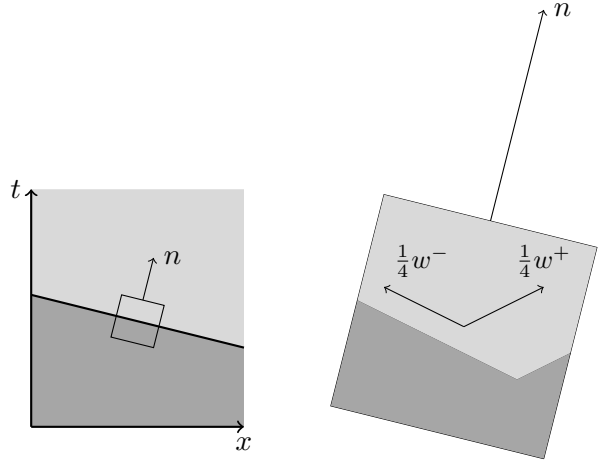

The construction can be easily generalized to the case $d>1$.
From the lower bound in   Theorem \ref{theorem:sufficient}\ref{item:simpletok}
we know that the best interfacial energy which can be attained using simple profiles is given by 
\begin{equation*}
 K_s= {\alpha_u} \left|[u]\right| +
   {\alpha_v} \left|v^+- a\right|+  {\alpha_v} \left|v^--a\right|
 + {\alpha_F} \left|[u]\right|  \left|n\cdot (1,a)\right|
\end{equation*}
for some $a\in\R^d$.
We choose  as above a large velocity $v\ge 1$, set 
 ${\alpha_F}={\alpha_v}={\alpha_u}=1$, $v^\pm=\pm v e_1$, $w^{\pm} = \left(1,v^{\pm}\right)$, $u^\pm=\pm1$, and 
pick a normal which corresponds to motion with velocity $-\tilde v e_1$, namely,
$n=(\tilde v,1, 0, \dots, 0)/\sqrt{1+\tilde v^2}$, for some $\tilde v>v$. 
The optimal $a$ is also parallel to $e_1$, and a short computation shows that it equals $v^-$, so that
\begin{equation*}
  K_s =  {\alpha_u} \left|[u]\right| +
   {\alpha_v} \left|v^-- v^+\right|+\alpha_F |[u]| |n\cdot w^-|\,.
\end{equation*}
The one-dimensional result shows that there is $N^+\in\R^2$ such that
$\overline K(N^+)<K_s$. This result can be immediately embedded in the higher-
dimensional setting by taking $l=2$, $N^1=(N^+,0,\dots, 0)$, $N^2=n-N^1$, 
$v^1=v^+$, $v^2=v^-$ in the upper bound of Theorem \ref{theorem:bounds}\ref{item:ddbounds}.  Therefore the same values, $v=2$ and $\tilde v\geq 4$, will do.

We refer to Section \ref{sec:illustration} for a visualization of a more general microscopic pattern in the case $d=2$.

\subsubsection*{Potential impact on the numerical implementation of the motion estimation model.}
The model studied here has a built--in consistency with respect to the local shading or texture information \emph{and} to the global geometry $J_u$
of moving and deforming objects in space--time, and is therefore attractive for 
concrete applications to imaging, based on an appropriate numerical implementation. 
The possible appearance of microstructures is, however, a very problematic feature of the model.
If microstructure appears the usability of the model in imaging is questionable. 

The analysis presented in this paper shows that microscopic patterns appear only if the  motion velocity $v$ encoded in the edge set $J_u$ is 
relatively large and the motion data  $v^+$ and $v^-$ encoded in the shading or the texture on both sides of the edge are substantially inconsistent  with $v$. 
Thus, in imaging applications the criteria for simple profiles stated in Theorem \ref{theorem:sufficient} and discussed at the beginning of this section
mostly rule out the appearance of microstructures and indicate that a one-scale method should be appropriate in normal situations.
The criteria  in Theorem \ref{theorem:sufficient} can be used in a numerical algorithm as an a-posteriori test for the 
appropriateness of the one-scale model based on upscaled versions of $u^+$, $u^-$, $v^+$, $v^-$, and $n$.  
If this test fails, it may be advisable to modify the the parameters $\alpha_u$, $\alpha_v$ and $\alpha_F$,
in ways which are suggested by the conditions in Theorem \ref{theorem:sufficient}.
Failure of the test, if one does not appropriately modify the parameters of the problem, 
leads to microscopic patterns. In particular,  one would expect to observe oscillations
in the single--scale numerical approximation of $u$ and $v$ in the vicinity of the edge set $J_u$, with a length
scale given by the spatial discretization.
In fact, the constructions used in the proofs of the upper bounds could be used to set up
a reduced microscopic model in a two--scale discretization approach, which would include 
a local optimization over a small set of  parameters describing the microstructure.
Numerical two-scale methods of this type are of theoretical interest, but in this concrete application they are of no practical relevance.

For a practical implementation of the single--scale model it is useful to observe that the functional $F$ is not jointly convex in $u$ and $v$. 
However, $F$ is separately convex in $u$ and in $v$ and thus classical primal-dual methods can be applied in an alternating minimization scheme.

\section{Proofs of the main results} \label{sec:proofs}
In this Section  we prove Theorem \ref{theorem:bounds} and Theorem \ref{theorem:sufficient}. We start by discussing the symmetries of  the function $K$ and 
reformulating  its definition on a different space of test functions 
in  Section \ref{sec:prelim}.
Moreover, we prove classical subadditivity and convexity properties for $K$.
Then, in Section \ref{sec:upper} we discuss explicit constructions which lead to upper bounds of $K$ for $d=1$
 and prove a lower bound of $K$ for $d=1$.
 Based on the ingredients of  Section \ref{sec:prelim} the proofs of the theorems are then 
given in Section \ref{sec:thmproofs}.
Finally, in Section \ref{sec:illustration} we illustrate the possible microscopic patterns which might arise in critical regimes complementing the discussion in Section \ref{sec:consequences}.
\subsection{Preliminaries}\label{sec:prelim}
    
In order to simplify the following discussion we first list  the symmetries
which $K$ obeys.
\begin{lmm}\label{lemma: invariance}
$K$ has the following symmetries:
\begin{enumerate}
\item[] $K\left(u^{+},u^{-},v^{+},v^{-},n\right) = K\left(u^{+}-u^{-}, 0 , v^{+}, v^{-}, n\right)$, \\[-2ex]
\item[] $K\left(u^{+},u^{-},v^{+},v^{-},n\right) = K\left(-u^{+},-u^{-},v^{+},v^{-},n\right)$,\\[-2ex]
\item[] $K\left(u^{+},u^{-},v^{+},v^{-},n\right) = K\left(u^{-},u^{+},v^{-},v^{+},n\right)$,\\[-2ex]
\item[] $K\left(u^{+},u^{-},v^{+},v^{-},n\right) = K\left(u^{-},u^{+},v^{-},v^{+},-n\right)$.
\end{enumerate}
\end{lmm}

\begin{proof} 
They all follow immediately from the definition. 
For the first equality replace $u$ by  $u - u^-$. 
For the second item replace $u$ by $- u$. 
For the third equality replace $u$ and $v$ by $u(-y)$ and $v(-y)$.
The last invariance is simply a relabeling of the boundary conditions.
\end{proof}

One key observation is that we can  replace the set of functions in the
definition of $K$ by a simpler class of functions.  In particular, the
function $u$ can be assumed to take only two values, and the velocity field
$v$ can be assumed to be smooth. This allows to give a classical sense to the
term $w\cdot Du$.
We stress that existence of minimizers is not expected in this restricted class,
indeed functions of this class are later  interpreted as microstructures.
 
\begin{lmm}\label{lem:coarea}
In the definition of $K$ in (\ref{eq:definitionK}) the set $\mathcal{A}$ can be replaced by 
\begin{align*}
\mathcal{A}^C= \Bigg\{& (u,v) \in BV\left(Q_n;\left\{u^+,u^-\right\}\right) \times C^{\infty}\left(\overline Q_n;\R^{d}\right): u = u^{\pm} \text{ and } v = v^{\pm} \text{ on } \\
&\partial Q_{n} \cap \left\{ y \cdot n = \pm \frac{1}{2} \right\} \text{ and periodic in the } m^i\text{-directions} \Bigg\}.
\end{align*}
\end{lmm}

\begin{proof}
We prove the claim in three steps. If $[u]=0$ then one immediately obtains
$K=\alpha_v|[v]|$ from which the result follows easily by standard density results. Hence we only need to deal with the case $[u]\ne0$.

Step 1. First we show that
   $\inf E[\mathcal{A}^\infty,Q_n] \leq \inf E[\mathcal{A},Q_n]$, where
\begin{equation*}
\mathcal{A}^{\infty} = \left(C^{\infty}\left(\overline{Q_{n}}; \mathbb{R}\right) \times C^{\infty}\left(\overline{Q_{n}}; \mathbb{R}^d\right) \right) \cap \mathcal{A}.
\end{equation*}
To prove the inequality, we choose $(u,v) \in \mathcal{A}$ and can assume that $\|v\|_{L^\infty}\le M$.
We extend $u$ and $v$
 periodically to $\left\{ y \in\R^{d+1}:\left|y \cdot n\right| < \frac12 \right\}$ and constantly in the $n$-direction and
define $\left(\overline{u}, \overline{v}\right)(y) = (u,v)(2y)$. 
Then a straightforward computation shows that
\begin{equation*}
 E\left[ \overline{u},\overline{v}, Q_n \right] = E\left[ u,v, Q_n \right].
\end{equation*}
For $0< \varepsilon <\frac14$ we define
$\left(u_{\varepsilon}, v_{\varepsilon} \right) = \left(\overline{u}, \overline{v}\right) * \rho_{\varepsilon}$ where $\rho_{\varepsilon}\in C^\infty_c(B_\varepsilon)$ is a standard mollifier.
Since $\left( \overline{u}, \overline{v} \right)$ is constant on $\left\{y \in\R^{d+1}:  y \cdot n  > \frac14 \right\}$ and $\left\{y \in\R^{d+1}:  y \cdot n  < -\frac14 \right\}$,
the functions $\left( u_{\varepsilon}, v_{\varepsilon} \right)$ satisfy $\left( u_{\varepsilon}, v_{\varepsilon} \right) \in \mathcal{A}^\infty$.

To prove convergence of the energy we first observe that  
by the general properties of mollification of Sobolev functions we immediately
 obtain 
 $u_{\varepsilon} \rightarrow \overline{u}$ and 
$v_{\varepsilon} \rightarrow \overline{v}$ in $W^{1,1}\left(Q_n;\R\right)$
and $W^{1,1}\left(Q_n;\R^d\right)$, respectively. 
The more subtle term is $\int_{Q_n} \left| w_{\varepsilon} \cdot \nabla u_{\varepsilon}
\right|\d y$, where $w_{\varepsilon} = \left(1, v_{\varepsilon} \right)$.
Since 
 $\left\| w_{\varepsilon} \right\|_{L^{\infty}} \leq \left\| \overline{w} \right\|_{L^{\infty}}\le M+1$
we can estimate
\begin{equation*}
\left\| w_{\varepsilon} \cdot \nabla u_{\varepsilon} \right\|_{L^1\left(Q_n\right)} \leq 
\left\| w_{\varepsilon} \cdot \left( \nabla u_{\varepsilon} - \nabla \overline{u}\right)  \right\|_{L^1\left(Q_n\right)} 
+  \left\| \left(w_{\varepsilon} - \overline{w}\right) \cdot \nabla \overline{u} \right\|_{L^1\left(Q_n\right)} 
+ \left\| \overline{w} \cdot \nabla \overline{u} \right\|_{L^1\left(Q_n\right)}.
\end{equation*}
The first term on the right hand side can be estimated by
\begin{equation*}
\left\| w_{\varepsilon} \cdot \left( \nabla u_{\varepsilon} - \nabla \overline{u}\right)  \right\|_{L^1\left(Q_n\right)} 
\leq \left\| \overline{w} \right\|_{L^{\infty}\left(Q_n\right)} \left\| \nabla u_{\varepsilon} - \nabla \overline{u} \right\|_{L^1\left(Q_n\right)} \rightarrow 0.
\end{equation*}
For the second term we observe that $w_{\varepsilon} - \overline{w} \rightarrow 0$ pointwise a.e. and that 
\begin{equation*}
\left| \left(w_{\varepsilon} - \overline{w}\right) \cdot \nabla \overline{u}\right| \leq 2 \left\| \overline{w} \right\|_{L^{\infty}\left(Q_n\right)} \left| \nabla \overline{u} \right| \text{ pointwise}.
\end{equation*}
Hence by Lebesgue's dominated convergence theorem we derive
\begin{equation*}
\lim_{\varepsilon \to 0} \left\| \left(w_{\varepsilon} - \overline{w}\right) \cdot \nabla \overline{u} \right\|_{L^1\left(Q_n\right)} = 0.
\end{equation*}
Therefore 
\begin{equation*}
\inf E[\mathcal{A}^\infty,Q_n]\le
\limsup_{\varepsilon \to 0} E\left[u_{\varepsilon}, v_{\varepsilon}, Q_n \right] \leq E\left[\overline{u}, \overline{v}, Q_n \right] =E\left[u, v, Q_n \right]
\end{equation*}
concludes the proof of the first step.

Step 2. We show that
   $\inf E[\mathcal{A}^C,Q_n]\le \inf E[\mathcal{A}^\infty,Q_n]$.
Let  $(u,v) \in \mathcal{A}^\infty$, assume for definiteness that  $u^+ \geq u^-$.
The coarea formula yields
\begin{equation*}
\int_{Q_n} \left( {\alpha_u} \left| \nabla u \right| + {\alpha_F} \left|w \cdot \nabla u \right|\right) \d y  \geq \int^{u^+}_{u^-} \left( \int_{Q_n} {\alpha_u} \left| D \chi_{\{u<t\}} \right| + {\alpha_F} \left| w \cdot D \chi_{\{u < t\}} \right| \right) \d t.
\end{equation*}
Indeed, the first term is standard, the second one follows easily from 
\cite[(3.33) in Th. 3.40]{AmFuPa00} approximating $w$ uniformly with piecewise constant vector fields.

Thus there exists $t^* \in \left(u^-,u^+\right)$ such that 
\begin{equation*}
\int_{Q_n} \left({\alpha_u} \left| \nabla u\right| + {\alpha_F} \left| w \cdot \nabla u\right|\right) \d y \geq 
\left(u^+ - u^-\right) \int_{Q_n}  {\alpha_u} \left| D \chi_{\left\{u < t^*\right\}} \right| + {\alpha_F} \left| w \cdot D \chi_{\left\{u < t^*\right\}} \right|.
\end{equation*}
We define  $Q^-_n = \left\{ y \in Q_n: u(y) < t^* \right\}$ and $u^* = u^+ + \left(u^- - u^+\right) \chi_{Q^-_n}$. Then
\begin{equation*}
\int_{Q_n} {\alpha_u}  \left| \nabla u \right| + {\alpha_F} \left| w \cdot \nabla u\right| \text{ d}y \geq \int_{Q_n} {\alpha_u} \left| D u^* \right| + {\alpha_F} \left|w \cdot D u^*\right|
\end{equation*}
where $Du^*$ on the right hand side has to be interpreted as a measure.
Since the function  $w$ is smooth, it is in particular continuous on
the jump set $J_{u^*}$, hence the integral is well defined.
Furthermore by the trace theorem $u^*$ fulfills the boundary conditions.
 Therefore $(u^*,v)\in \mathcal{A}^C$ and   $\inf E[\mathcal{A}^C,Q_n]\le \inf E[\mathcal{A}^\infty,Q_n]$. 

Step 3. We prove that    $\inf E[\mathcal{A},Q_n]\le\inf E[\mathcal{A}^C,Q_n]$. 
As in Step 1 we choose $(u,v)\in \mathcal{A}^C$, extend both functions periodically,
scale  them  to $(\overline u,\overline v)$,
 and mollify $\overline u$ (but not $\overline v$) to obtain $u_\varepsilon=\overline u \ast \rho_\varepsilon$.
For the same reasons as in Step 1
we have  $(u_\varepsilon, \overline v)\in \mathcal{A}$ and
\begin{equation*}
  \lim_{\varepsilon\to0} 
\int_{Q_n} |\nabla u_\varepsilon| \d y=\int_{Q_n} |D \overline u|=\int_{Q_n} |D u| \,.
\end{equation*}
It remains to show that $\int_{Q_n} |\overline w\cdot\nabla u_\varepsilon|\d y\to
 \int_{Q_n} |\overline w\cdot D\overline u|$, where as usual
 $\overline w=(1,\overline v)$.
To see this, we deduce from $|D\overline u|(\partial Q_n)=0$ for the Radon measure 
$\overline w\cdot D\overline u$ that (cf.~ \cite[Prop.~3.7]{AmFuPa00})
\begin{equation*} 
  \lim_{\varepsilon\to0} 
\int_{Q_n} |\rho_\varepsilon\ast (\overline w\cdot D\overline u)| \d y=
\int_{Q_n} |\overline w\cdot D\overline u|\,.
\end{equation*}
Furthermore, 
\begin{equation*}
\int_{Q_n}  [
\rho_\varepsilon\ast (\overline w\cdot D\overline u)-
\overline w\cdot (\rho_\varepsilon\ast D\overline u) ]\d y
=
\int_{Q_n} \left[ \int_{B_\varepsilon(y)} \rho_\varepsilon(y-z)
 (\overline w(z)-\overline w(y))\cdot D\overline u(z)\right] \d y
\end{equation*}
converges to zero because $\overline w$ is uniformly continuous on $Q_n$.
Hence $$\limsup_{\varepsilon\to0} E[u_\varepsilon,\overline v,Q_n]\le E[\overline u,\overline v,Q_n]=E[u, v,Q_n]\,,$$ which proves the claim.
\end{proof}

Next we 
 give an iterated relaxation formula.
 \begin{lmm}\label{lemmarecrelax}
 For $(u,v)\in \mathcal{A}^C$ we have
   \begin{equation*}
 K(u^{+}, u^{-},v^{+},v^{-}, n) = \inf \left\{ 
E^*[u,v,Q_{n}]: (u,v) \in \mathcal{A}^C \right\}\,,\end{equation*}
where $E^*[u,v,Q_n]= \int_{Q_n} {\alpha_v} |Dv| +\int_{Q_n\cap J_u} K(u^+,u^-,v,v,\nu) d\mathcal{H}^d$ for $(u,v) \in \mathcal{A}^C$ with
$\nu$ being the normal to $J_u$.
 \end{lmm}
 \begin{proof}
   Let $E^*$ be the relaxation of $E$, defined as in (\ref{eq:Fstar}). From the general relaxation result 
we know that it has a form corresponding to (\ref{eqfstar}), namely,
\begin{alignat*}1
E^*[u,v,Q_n]=&
\int_{Q_n} g\left(v,\nabla u, \nabla v\right) \text{ d}y 
+ \int_{Q_n} g\left(v,D^c u, D^c v\right)\\
& + \int_{J_{(u,v)}\cap Q_n}  K\left(u^+,u^-,v^+,v^-,\nu\right) \text{ d}\mathcal{H}^{d}\,.
 \end{alignat*}
If  $(u,v)\in \mathcal{A}^C$ in particular
$E^*[u,v,Q_n]= \int_{Q_n} {\alpha_v} |Dv| +\int_{Q_n\cap J_u} K(u^+,u^-,v,v,\nu) d\mathcal{H}^d$, as given in the statement.

Fix $u^+$, $u^-$, $v^+$, $v^-$, $n$; let $(u,v)\in \mathcal{A}^C$, extend $u,v$
periodically and define as in Step 1 of the proof of Lemma \ref{lem:coarea}
$(u_j,v_j)=(u,v)(jx)$. Then $(u_j,v_j)\to (u_*, v_*)=(u^-,v^-)+([u],[v])
\chi_{y\cdot n>0}$ in $L^1$. Hence, by the lower semicontinuity of $E^*$ we obtain
\begin{equation*}
  K(u^{+}, u^{-},v^{+},v^{-}, n) = E^*[u_*, v_*,Q_n]\le \liminf_{j\to\infty} E^*[u_j, v_j,Q_n]=E^*[u,v,Q_n]\,.
\end{equation*}
The other inequality follows immediately from Lemma \ref{lem:coarea} and $E\geq E^*$.
 \end{proof}

 Next we prove two more general properties, namely, subadditivity and convexity of the surface energy $K$, which are well 
known for example in the setting of variational problems on partitions \cite{AmbrosioBraides1,AmbrosioBraides2}.
We start with subadditivity and its main consequence in the present setting.

 \begin{lmm}[Subadditivity]\label{lemmasubaddit}
Let $u^+,u^- \in \R$, $v^+,v^- \in \R^d$, $n\in S^d$. Then:
\begin{enumerate}
\item\label{lemmasubaddit1} For any $u'\in \R$, $v'\in\R^d$ one has
  \begin{equation*}
K(u^{+}, u^{-},v^{+},v^{-}, n)\le K(u^{+}, u',v^{+},v', n)
+K(u', u^{-},v',v^{-}, n)\,.
      \end{equation*}
\item\label{lemmasubaddit2} For any $a\in \R^d$ one has
  \begin{equation*}
K(u^{+}, u^{-},v^{+},v^{-}, n)\le {\alpha_u} |[u]|+{\alpha_v} (|v^+-a|+|v^--a|)
+{\alpha_F}|[u]|\, |(1,a)\cdot n|\,.
  \end{equation*}
\end{enumerate}
 \end{lmm}
 \begin{proof}
   \ref{lemmasubaddit1}: For $j\in\N$ we define
   \begin{equation*}
     (u_j,v_j)(y)=
     \begin{cases}
       (u^+,v^+) & \text{ if } y\cdot n>\frac1j\\
       (u',v') & \text{ if } -\frac1j\le y\cdot n\le\frac1j\\
       (u^-,v^-) & \text{ if } y\cdot n<-\frac1j
     \end{cases}\,.
   \end{equation*}
   Clearly $(u_j,v_j)\to(u_\infty,v_\infty)=(u^-,v^-)+([u],[v]) \chi_{y\cdot n>0}$ in $L^1$ 
and $E^*[u_j,v_j,Q_n]=K(u^{+}, u',v^{+},v', n)+K(u', u^{-},v',v^{-}, n)$ for all $j>2$.
   Since $E^*$ is lower semicontinuous, we obtain
   \begin{alignat*}1
     K(u^+,u^-,v^+,v^-,n)&=E^*[u_\infty,v_\infty,Q_n]
     \le\liminf_{j\to\infty}E^*[u_j,v_j,Q_n] \,,
   \end{alignat*}
which concludes the proof.

   \ref{lemmasubaddit2}: Two applications of \ref{lemmasubaddit1} and the fact that $K=\alpha_v |[v]|$ if $|[u]|=0$ give
   \begin{alignat*}1
     K(u^{+}, u^{-},v^{+},v^{-}, n)&\le 
     K(u^+, u^+,v^+,a, n)+
     K(u^+, u^-,a,a, n)+     K(u^-, u^-,a,v^-, n)\\
&\le {\alpha_v} |v^+-a| + ({\alpha_u} |[u]|
+{\alpha_F}|[u]|\, |(1,a)\cdot n|)+
{\alpha_v}|a-v^-|\,.
   \end{alignat*}
 \end{proof}

We next show that $K$ is convex in the last argument, after having been extended to a positively one-homogeneous function.
This will be linked to the possible development of oscillations of the interface.
 \begin{lmm}[Convexity]\label{lemmacomvex}
For given $u^+,u^-\in\R$, $v^+,v^-\in\R^d$ we define $h:\R^{d+1}\to\R$ by
\begin{equation*}
  h(n)=|n| K(u^+,u^-,v^+,v^-,\frac{n}{|n|})\,.
\end{equation*}
The function $h$ is positively one-homogeneous and convex,
in particular, 
\begin{equation}\label{eqhsubaddit}
h(n^++n^-)\le  h(n^+)+h(n^-)\hskip1cm \text{ for all $n^+,n^-\in \R^{d+1}$}.
\end{equation}
 \end{lmm}
 \begin{proof}
   Positive one-homogeneity is obvious from the definition. It implies
$\lambda h(n^+)+(1-\lambda) h(n^-)=h(\lambda n^+)+h((1-\lambda) n^-)$, hence convexity is equivalent to subadditivity; therefore it suffices
 to prove the inequality  (\ref{eqhsubaddit}).
Furthermore, it suffices to consider the case that
$|n^++n^-|=1$ with $n^+$ and $n^-$ linearly independent (the case that $n^+$ and $n^-$ are parallel, including the case that one vanishes,  is immediate).

  \begin{figure}[rotate = 300]
\centering
\begin{tikzpicture}[rotate = 300, scale =0.8]
 \foreach \x in {0,1,2}
\draw[red] (1-2*\x, 1+2*\x) -- (-0.5-2*\x, 1.5+2*\x);

\foreach \x in {0,1,2}
\draw[blue] (-0.5-2*\x, 1.5+2*\x) -- (-1-2*\x, 3+2*\x);

\node at (-3.5,4) {$\omega^+$};
\node at (-1.5,4) {$\omega^-$};
 \end{tikzpicture}
 \quad
 \begin{tikzpicture}[rotate = 300, scale =0.8]

\fill[gray] (-1,3) -- (-2,2) -- (-4,4) -- (-3,5) -- (-2.9,4.7) -- (-2.5,3.5) -- (-1.3,3.1) -- (-1,3);
\fill[gray!50] (-1,3) -- (0,4) -- (-2,6) -- (-3,5) -- (-2.9,4.7) -- (-2.5,3.5) -- (-1.3,3.1) -- (-1,3);

 \foreach \x in {0,1,2}
\draw[red] (1-2*\x, 1+2*\x) -- (-0.5-2*\x, 1.5+2*\x);

\foreach \x in {0,1,2}
\draw[blue] (-0.5-2*\x, 1.5+2*\x) -- (-1-2*\x, 3+2*\x);

\draw (-2,2) -- (-4,4) -- (-2,6) -- (0,4) -- (-2,2);
\draw[->] (-1,5) -- (1,7);
\draw(1,7) node[black, anchor=north]{$n$};

\draw(-1.55,5.55) node[black, anchor= north, rotate=75]{$u=u^+$};
\draw(-3.55,3.55) node[black, anchor= south, rotate=75]{$u=u^-$};

\end{tikzpicture}

\caption{The construction of $\gamma$ is a periodic laminate using $\gamma^+$ and $\gamma^-$ that splits $\mathbb{R}^{2}$ into two connected components $\omega^+$ and $\omega^-$. Then $u$ is defined as $u^-$ on $\omega^-$
and as $u^+$ on $\omega^+$.}
\label{figure:constructiongamma}
\end{figure}
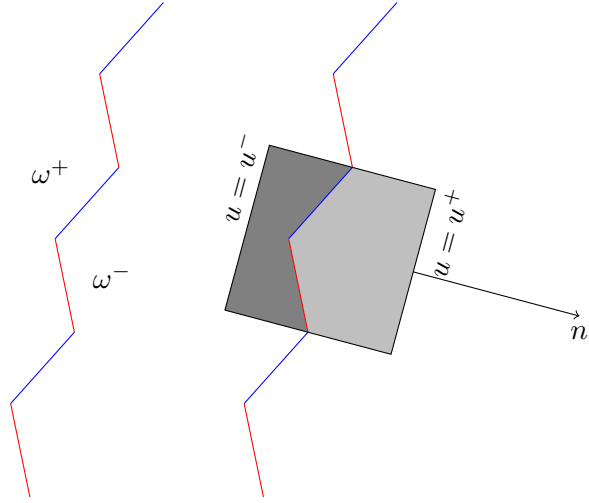

   By the scaling argument used in the proof of Lemma \ref{lemmarecrelax}
   it suffices to construct functions $(u,v): \R^{d+1}\rightarrow \R^{d+1}$ which are
   periodic in the $d$ directions orthogonal to $n = n^+ + n^-$ and
   with $(u,v)(y)=(u^\pm,v^\pm)$ whenever $\pm y\cdot
   n > L$, for some $L>0$. 
   We shall first construct a curve $\gamma$ that can be interpreted as the jump set of a
   function $u$ satisfying the boundary conditions. 
   The curve $\gamma$ will consist of two subsets, one has 
   normal  $n^+/|n^+|$ and measure $\left|n^+\right|$, the other has normal  $n^-/|n^-|$ and measure
   $\left|n^-\right|$. To make this precise, 
we define $n_p^\pm=Jn^\pm$,  where  $J\in SO(d+1)$ is a 90-degree rotation in the plane spanned by $n^+$ and $n^-$.
We define 
  \begin{equation*}
    \gamma^+ = [0,1) n^+_p \hskip2mm\text{ and } \hskip2mm\gamma^- = 
n^+_p +[0,1)n^-_p\,,
  \end{equation*}
so that
  $\gamma^-$ ends at $n^+_p+n^-_p=Jn$, with $n=n^++n^-$. 
We set $m^1=Jn$, choose $m^2, \dots, m^d$  such that $\{n,m^1,\dots, m^d\}$ is an orthonormal basis of $\R^{d+1}$, and define
 $\Sigma = \Z m^1 +(\gamma^+ \cup
  \gamma^-) + \sum_{i\ge 2} m^i\R$, see Fig.~\ref{figure:constructiongamma}
 (notation: $A+cB=\{a+cb: a\in A, b\in B\}$). This is a $d$-dimensional surface, corresponding to the constant extension in the directions $m^2,\dots,
 m^d$ of the curve obtained joining alternatively copies of the segments
$\gamma^+$ and $\gamma^-$.
 In particular, $\R^{d+1} \setminus \Sigma$ has exactly two connected components;
  one of them, call it $\omega^-$, contains $\{y\cdot n<-L\}$ and the other,
  call it $\omega^+$,   contains $\{y\cdot n>L\}$, 
where $L=|n^+|+|n^-|+1$, see Fig.~\ref{figure:constructiongamma}.   We set 
$(u,v)=(u^-,v^-)+([u],[v]) \chi_{\omega^+}$.
  Then the normal to the jump set is $n^+/|n^+|$ on a set of 
measure $|n^+_p|=|n^+|$, and correspondingly for $n^-$.
We define by rescaling $(u_j,v_j)(y)=(u,v)(jy)$, so that
  $(u_j,v_j)\to (u_*, v_*)=(u^-,v^-)+([u],[v]) \chi_{y\cdot n>0}$ in $L^1$.
 Therefore, dropping the arguments $u^+,u^-,v^+,v^-$ for brevity,
  \begin{alignat*}1
    E^*[u_*,v_*,Q_n]=  h(n^++n^-)&= K(n^++n^-)\\
&\le |n^+| K(\frac{n^+}{|n^+|})
    +|n^-| K(\frac{n^-}{|n^-|})     =h(n^+)+h(n^-).
  \end{alignat*}
 \end{proof}

\subsection{Upper and lower bounds for $d=1$}\label{sec:upper}

 \begin{prpstn}[Construction for $d=1$]\label{propupperbound1d}
Let $d=1$ and $u^+,u^-,v^+,v^- \in \R$, $n \in S^1$. Then  for any $N^+\in\R^2$ we have
\begin{equation*}
K(u^{+}, u^{-},v^{+},v^{-}, n)\!\le\! \left({\alpha_u}\! \left|[u]\right| \!+\! {\alpha_v}\!
\left|[v]\right|\right)\!
\left(\left|N^+\right| \!+\! \left|N^-\right|\right)  
 + {\alpha_F} \left|[u]\right| 
\left( \left|N^+ \!\!\cdot w^+\right| +\left| N^- \!\!\cdot w^-\right| \right),
\end{equation*}
where $N^- = n - N^+$ and $w^\pm=(1,v^\pm)$.
 \end{prpstn}

 \begin{proof}
The result follows by Lemma \ref{lemmacomvex} and Lemma \ref{lemmasubaddit}\ref{lemmasubaddit2}, taking once $a=v^+$ and once $a=v^-$. For the sake of illustration we give here a self-contained, explicit construction.

We start by replicating the construction of  Lemma \ref{lemmacomvex}.
Again, we  first construct a curve $\gamma$ that will be interpreted as the jump set of a function $u$ satisfying the boundary conditions, and which 
consists of two subsets, with the normals
  $N^+/|N^+|$ and  $N^-/|N^-|$. Precisely, we
 define as in the proof of Lemma \ref{lemmacomvex}
  \begin{equation*}
    \gamma^+ = [0,1) (N^+)^\perp \hskip2mm\text{ and } \hskip2mm\gamma^- = 
(N^+)^\perp +[0,1)(N^-)^\perp\,,
  \end{equation*}
so that
  $\gamma^-$ ends at $(N^++N^-)^\perp=n^\perp$.
We set $\gamma = n^\perp \Z+(\gamma^+ \cup
  \gamma^-)$  and
  $L=|N^+|+|N^-|+1$. 
  Then $\R^2 \setminus \gamma$ has exactly two connected components and for $L$ sufficiently large
  one of them, call it $\omega^-$, contains $\{y\cdot n<-L\}$ and the other,
  call it $\omega^+$,   contains $\{y\cdot n>L\}$, see Figure \ref{figure:constructiongamma}.   We set $u=u^-+(u^+-u^-) \chi_{\omega^+}$.
  Then  $u$ fulfills the desired boundary conditions.

  Ideally the function $v$ should equal $v^+$ on $\omega^+ \cup \left(\gamma^++n^\perp\Z\right)$,  
  and $v^-$ on  $\omega^- \cup \left(\gamma^-+n^\perp\Z\right)$. 
As an approximation we construct a function $v_{\varepsilon}$ as follows  (cf. Figure \ref{figure:
    upperbound1dv}).
 Precisely,   for $\varepsilon\in(0,1)$, we set
  \begin{equation*}
    \omega^+_\varepsilon = \omega^+ \cup B_\varepsilon(\gamma^++n^\perp\Z)
\setminus B_\varepsilon(\gamma^-+n^\perp\Z)
  \end{equation*}
and $\tilde v_\varepsilon= v^-+ (v^+-v^-)
  \chi_{\omega^+_\varepsilon}$. It is easy to see that $|D\tilde
  v_\varepsilon|(Q_n)\le |v^+-v^-| (|N^+|+|N^-|+4\pi\varepsilon)$. 
  Let now $v_\varepsilon=\varphi_\varepsilon\ast \tilde v_\varepsilon\in
  C^\infty$, where $\varphi_\varepsilon\in C^\infty_c(B_\varepsilon)$ is a
  mollification kernel. Then (cf. Figure \ref{figure: upperbound1dv}) there exists a constant $c$, such that
  \begin{enumerate}
  \item $v_\varepsilon = v^-$ on $\gamma^-$; $\mathcal{H}^1(\gamma^+\setminus\{ v_\varepsilon = v^+\}) \le c\varepsilon$,
   \item $\int_{Q_n} \left|\nabla v_\varepsilon\right| \text{ d}y \le
 \left|[v]\right| \left( \left|N^+\right| + \left|N^-\right|
 + c \varepsilon\right) $,
   \item $v_\varepsilon=v^\pm$ on $\{\pm y\cdot n >L\}$.
  \end{enumerate}
Thus   $(u,v_\varepsilon) \in \mathcal{A}^C$ and 
a straightforward computation shows that
    \begin{align*}
   E[u,v_\varepsilon,Q_n] = &\left( {\alpha_u} \left|[u]\right| + {\alpha_v} \left|[v]\right|\right) \left( \left|N^+\right| + \left|N^-\right|+c\varepsilon \right) 
   \\ &+ {\alpha_F} \left|[u]\right| \left( \left|N^+ \cdot w^+\right| + \left|w^- \cdot N^-\right|+c|[v]|\varepsilon \right).
  \end{align*}
  Finally, we can take $\varepsilon$ arbitrarily small, which establishes the claim.
   
\begin{figure}[t]
\floatbox[{\capbeside\thisfloatsetup{capbesideposition={right,top},capbesidewidth=8cm}}]{figure}[\FBwidth]
{\caption{The picture shows the construction of $u$ and $v$. 
On the dark part $u$ takes the value $u^-$, on the grey part the value $u^+$.
The brown line is the part of the jump set of $u$ parallel to $N^-$ ($\gamma^-$), the blue line represents the part of $\gamma$ with normal parallel to $N^+$ ($\gamma^+$). 
The black line indicates up to smoothing the jump set of $\tilde{v}_{\varepsilon}$. 
As can be seen this is in terms of length and normal essentially the jump set of $u$.}\label{figure: upperbound1dv}}
{\centering
\begin{tikzpicture}[rotate = 300]
\draw[white] (0.5, 0.5) -- (1, -0.5); \fill[gray] (1,1) -- (0,0) -- (-2,2) -- (-1,3) -- (-0.9,2.7) -- (-0.5,1.5) -- (0.7,1.1) -- (1,1);
\fill[gray!50] (1,1) -- (2,2) -- (0,4) -- (-1,3) -- (-0.9,2.7) -- (-0.5,1.5) -- (0.7,1.1) -- (1,1);

\draw (0,0) -- (-2,2) -- (0,4) -- (2,2) -- (0,0);
\draw[red] (1, 1) -- (-0.5, 1.5);
\draw[blue] (-0.5,1.5)-- (-1,3);

\draw[black] (1.05 - 1.05/2 + 1.7/2 - 1.75/6, 1.7 - 1.75/3 + 1.05/2 - 1.7/2 + 1.75/6) -- (1.05,1.7 - 1.75/3) -- (-0.7, 1.7) -- (-1.025, 2.675) -- (-2 + 1.05 - 1.05/2 + 1.7/2 - 1.75/6, 2 +  1.7 - 1.75/3 + 1.05/2 - 1.7/2 + 1.75/6);
\draw[black,dotted,line width=0.04cm] (-2 + 1.05 - 1.05/2 + 1.7/2 - 1.75/6, 2 +  1.7 - 1.75/3 + 1.05/2 - 1.7/2 + 1.75/6) -- (-2 + 1.05 - 1.05/2 + 1.7/2 - 1.75/6 - 1, 2 +  1.7 - 1.75/3 + 1.05/2 - 1.7/2 + 1.75/6 +1);
\draw[->] (1,3) -- (3,5);
\draw(3,5) node[black, anchor=north]{$n$};
\draw(-1.5,2.5) node [black, anchor=west, rotate=75]{$\tilde{v}_{\varepsilon}=v^-$};
\draw(-1.55,1.55) node[black, anchor= south, rotate=75]{$u=u^-$};

\draw(-0.5,3.5) node[anchor=west, rotate=75]{$\tilde{v}_{\varepsilon}=v^+$};
\draw(0.45,3.55) node[black, anchor= north, rotate=75]{$u=u^+$};

\end{tikzpicture}
\vspace{3ex}}

\end{figure}
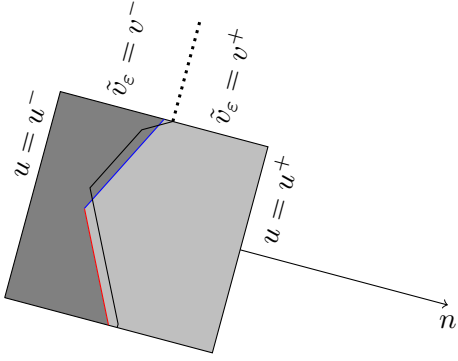  
 \end{proof}

In order to prove the lower bound in Theorem \ref{theorem:bounds} \ref{item:1dbounds} we need the following truncation lemma. 
As it is frequently the case in truncation, this can only be done if the relevant field (the velocity here) is scalar, and the result is therefore restricted to $d=1$.

\begin{lmm}[Truncation for $d=1$]\label{truncation1d}
Let $d=1$ and  $u^+,u^-,v^+,v^- \in \R$ be such that $v^- \leq v^+$. Let $n \in S^1$, $(u,v) \in \mathcal{A}^C$ and 
$
\overline{v}: Q_n \rightarrow \left[v^-,v^+\right]
$
be defined as 
\begin{equation*}
\overline{v}(y) = \begin{cases}
v(y) &\text{ if }v^- \leq v(y) \leq v^+\,, \\
v^+ &\text{ if } v^ +< v(y) \,, \\
v^- &\text{ if } v(y) < v^-.
\end{cases}
\end{equation*}
If $\left|[u]\right| \leq 2\frac{{\alpha_v}}{{\alpha_F}}$, then we have
\begin{equation}
{\alpha_v} \int_{Q_n}  \left| \nabla v\right| \text{ d}y + {\alpha_F} \int_{Q_n} w \cdot D
 u \geq {\alpha_v}\int_{Q_n}  \left| \nabla \overline{v} \right|\text{ d}y + {\alpha_F} 
\int_{Q_n} \overline{w} \cdot D u \,,
\end{equation}
where $w=(1,v)$ and $\overline w=(1,\overline v)$.
\end{lmm}
\begin{proof}
 Let $Q_n^- = \left\{ y\in Q_n:u(y) = u^- \right\}$, $\gamma = \left(\del Q_n^-\right) \cap Q_n$ and $n^{\gamma}$ be the outer normal to $Q_n^-$ on $\gamma$, 
so that $Du=[u]\otimes n^\gamma d\mathcal{H}^1\LL \gamma$.
We want to estimate the term
\begin{equation*}
{\alpha_F}    \int_{Q_n}  \left(\overline w - {w}\right) \cdot D u = 
 {\alpha_F} [u] \int_{\gamma} \left(\overline w - {w}\right) \cdot n^{\gamma} \text{ d}\mathcal{H}^{1} \,.
\end{equation*}
Since $\overline w- w$ vanishes on $\{y\cdot n=-1/2\}$ and it is periodic in the  direction $n^\perp$ we have that 
$\int_{\partial Q_n^-\setminus\gamma} (\overline w- w)\cdot \nu 
\text{ d}\mathcal{H}^d=0$ (here $\nu$ is the outer normal to $Q_n^-$), and the same on $Q_n^+=Q_n\setminus Q_n^-$. Therefore
by the Gauss-Green formula
\begin{equation*}
\int_{\gamma} \left(\overline w - {w}\right) \cdot n^{\gamma} \text{ d}\mathcal{H}^{1}     = \int_{Q_n^-} \operatorname{div} (\overline w- w) \text{ d}y 
=-\int_{Q_n^+} \operatorname{div} (\overline w- w) \text{ d}y \,.
\end{equation*}
Estimating $|\operatorname{div} w|\le |\nabla v|$ and averaging over the two sides of $\gamma$ we obtain
\begin{align*}
 {\alpha_F}    \int_{Q_n} \left(\overline w - {w}\right) \cdot D u 
 &\leq {\alpha_F} \frac{\left|[u]\right|}{2} \int_{Q_n} \left| \operatorname{div}\left( \overline w - {w}\right) \right| \text{ d}y 
 \leq {\alpha_F} \frac{\left|[u]\right|}{2} \int_{Q_n} \left|\nabla \overline v - \nabla {v}\right| \text{ d}y \\
 &= {\alpha_F} \frac{\left|[u]\right|}{2} \left(\int_{Q_n} \left|\nabla v\right| \text{ d}y - \int_{Q_n} \left|\nabla \overline{v}\right| \text{ d}y\right) \\
 &\leq {\alpha_v} \int_{Q_n} \left|\nabla v\right| \text{ d}y - {\alpha_v} \int_{Q_n} \left|\nabla \overline{v}\right| \text{ d}y\,,
\end{align*}
where we used that locally either $v=\overline v$ and $\nabla v=\nabla\overline v$, or $\nabla \overline v=0$. This concludes the proof.
\end{proof}

At this point we present the key lower bound in the one dimensional case. The main idea is to reduce the energy to an integral on  the jump set of $u$ and separate it into different  parts, in order to show that the optimal curve has a structure similar
to the one used for the construction in the upper bound and illustrated in Figure \ref{figure:constructiongamma}. 
For this we start from the reduction to intensity maps $u$ which are characteristic functions, done in Lemma \ref{lem:coarea}, 
and the truncation of $v$, discussed in Lemma \ref{truncation1d}. 
Then we show that we can consider also functions $v$ which take only two values, and use this to subdivide  $\gamma$ 
into two subsets, a strategy similar to
the one used 
in \cite[Lemma 4.5]{ContiGarroniMassaccesi} to obtain 
the relaxation of a line-tension model for dislocations in crystals. 
\begin{prpstn}[Lower bound for $d=1$]
\label{proposition:lowerboundsamesigns1d}
Let $d=1$, $u^+,u^-,v^+,v^- \in \R$, $n \in S^1$ be such that $\left|[u]\right| \leq 2\frac{{\alpha_v}}{{\alpha_F}}$. 
Then $K(u^+,u^-,v^+,v^-,n) \geq \min_{N^+ \in \R^2} \underbar{K}(N^+)$ where
\begin{equation*}
\underline{K}\left(N^+\right) = \left({\alpha_u} \left|[u]\right| + {\alpha_v} \left|[v]\right|\right)\left( \left|N^+\right| + \left|n - N^+\right| \right) 
+ {\alpha_F} \left|[u]\right| \left|N^+ \cdot w^+ + \left(n - N^+\right) \cdot w^-\right|.
\end{equation*}
\end{prpstn}
Remark. If $\left(w^- \cdot n\right) \left(w^+ \cdot n\right)\le 0$ then the minimizer $N^+_*$ belongs to the segment $[0,n]$ and makes the last term vanish, the bound reduces to  $K\ge \underline{K}\left(N^+_*\right) = \left({\alpha_u} \left|[u]\right| + {\alpha_v} \left|[v]\right|\right)$.
\begin{proof}
By Lemma \ref{lemma: invariance} we can assume without loss of generality
that $v^- \leq v^+$.  Let $(u,v) \in \mathcal{A}^C$. 
By Lemma \ref{truncation1d} we may further assume that $v^-\le v\le v^+$.
 We write $u = u^+ - [u] \chi_{Q_n^-}$ for a set $Q_n^- \subset Q_n$ of finite perimeter.
 Using the abbreviation $\gamma = \left(\del Q_n^-\right) \cap Q_n$ and $n^{\gamma}$ for the outer normal to $Q_n^-$ on $\gamma$ we have as in the proof of Lemma \ref{truncation1d}
$Du=[u] n^\gamma \mathcal{H}^1\LL\gamma$ and
\begin{equation*}
E[u,v,Q_n] = {\alpha_u}\int_{Q_n}  \left| D u \right| + {\alpha_v} \int_{Q_n} \left| \nabla v\right| \d y
+ {\alpha_F} \int_{Q_n} \left| w \cdot D u\right|   \,.
\end{equation*}
We estimate 
\begin{equation*}
E[u,v,Q_n]\geq {\alpha_u} \left|[u]\right| \mathcal{H}^1(\gamma) 
+ {\alpha_v}  \left| \int_{Q_n^-}\nabla v \text{ d}y \right|
+ {\alpha_v}  \left| \int_{Q_n^+}\nabla v \text{ d}y \right|
+ {\alpha_F} \left|[u]\right|  \, \left|\int_{\gamma} w \cdot n^{\gamma} \text{ d}\mathcal{H}^1 \right|. 
\end{equation*}
This is the key estimate to establish the lower bound, and it is sharp if the three integrands have a constant orientation on the respective domains.

We observe that the quantity on the right-hand side depends only on the value of $v$ on $\gamma$. Indeed, since $v$ is periodic in the direction orthogonal to $n$
and $\int_{\gamma} n^{\gamma} \text{ d}\mathcal{H}^1 = n$,
\begin{equation*}
 \int_{Q_n^-} \nabla v \d y
= \int_\gamma v n^\gamma \d\mathcal{H}^1- v^- n
= \int_\gamma (v-v^-) n^\gamma \d\mathcal{H}^1
\end{equation*}
and the same on $Q_n^+$. 
We therefore define $G:L^\infty(\gamma,[v^-,v^+])\to\R$ by
\begin{equation*}
G[\tilde v]=
  {\alpha_v} \left| \int_\gamma (\tilde v-v^-) n^\gamma \d\mathcal{H}^1\right| 
+  {\alpha_v} \left| \int_\gamma (\tilde v-v^+) n^\gamma \d\mathcal{H}^1\right| 
+{\alpha_F} |[u]|  \left| n_0+
\int_\gamma\tilde  v n_1^\gamma \d\mathcal{H}^1\right| \,,
\end{equation*}
so that (we recall that $n_0$ is the time component, $n_1$ the space component)
\begin{equation*}
  E[u,v,Q_n]\ge {\alpha_u} |[u]|\mathcal{H}^1(\gamma)+\min G.
\end{equation*}

Existence of a minimizer 
in $L^\infty(\gamma,[v^-,v^+])$ follows immediately from the
convexity of the functional  $G$. Even more, 
the functional  $G$ is weakly continuous on $L^\infty(\gamma)$, 
and therefore for any $\varepsilon>0$ 
there is $v_*\in L^\infty(\gamma,\{v^-,v^+\})$
such that $G(v_*)\le \min G + \varepsilon$. This permits 
us to reduce to the situation where also the velocity $v$ 
takes only two values.

We split $\gamma$ into two subsets depending on the value of $v_*$:
\begin{equation*}
\gamma^{-} =\left\{y \in \gamma: v_*(y) =v^-\right\} \quad \text{and} \quad
\gamma^{+} =\left\{y \in \gamma: v_*(y)=v^+ \right\},
\end{equation*}
and define 
\begin{equation*}
N^- = \int_{\gamma^-} n^{\gamma} \text{ d}\mathcal{H}^1
\quad \text{and} \quad
N^+ = \int_{\gamma^+} n^{\gamma} \text{ d}\mathcal{H}^1\,.
\end{equation*}
These quantities characterize the effective orientation of $\gamma^+$ and $\gamma^-$, corresponding to  the two normals in
 Proposition  \ref{propupperbound1d}  and Lemma \ref{lemmacomvex}. Since $v_*$ is constant in each of the  
 two subsets of $\gamma$, by convexity one could assume that $n^\gamma$ is also constant on each of them, and then rearrange so that $\gamma$ is composed of two segments. 
This illustrates the idea behind the energy bound, however it is not needed to conclude the proof.

It suffices indeed to separate the integrals in the definition of $G$ into a $\gamma^+$ and a $\gamma^-$ part,
then a short computation gives
\begin{alignat*}1
  G[v_*]&={\alpha_v} |[v]| (|N^+|+|N^-|) + {\alpha_F} |[u]|
| n_0 + v^+ N^+_1 + v^- N^-_1|\\
&={\alpha_v} |[v]| (|N^+|+|N^-|) + {\alpha_F} |[u]|
| w^+\cdot N^++ w^-\cdot N^-|\,,
\end{alignat*}
where  $w^\pm=(1,v^\pm)$. Since 
$\mathcal{H}^1(\gamma)\ge |N^+|+|N^-|$ and  by periodicity
 $N^++N^-=n$ this gives
\begin{equation*}
E[u,v,Q_n]\ge  \underline{K}\left(N^+\right) -\varepsilon
\ge \min_{N^+ \in \R^2}\underline{K}\left(N^+\right) -\varepsilon
\end{equation*}
for all $(u,v)\in\mathcal{A}^C$ and  $\varepsilon>0$, which concludes the proof.
 \end{proof}

\subsection{Proofs of Theorem \ref{theorem:bounds} and \ref{theorem:sufficient}}\label{sec:thmproofs}

At first we prove a lower bound in the higher dimensional case.
\begin{prpstn}[Lower bound for $d\geq1$]\label{proposition:lowerbounddge1}
For $(u,v)\in \mathcal{A}^C$ let us define
 $Q_n^- = \left\{y \in Q_n: u(y) = u^-\right\}$ and  
$W = (W)_{\substack{i=1,\dots,d \\ j =0,\dots,d}}= \int_{J_u} v\otimes \nu \text{ d}\mathcal{H}^d\in\R^{d\times(d+1)}$ 
with $\nu$ denoting the outer normal of $Q_n^-$ on $J_u$. Then
$$
E[u,v,Q_n] \geq{\alpha_F} \left|[u]\right| n_0
 + {\alpha_u} \left|[u]\right| + f(W)\,,
$$
where $ f(W) = {\alpha_F} \left|[u]\right| \sum_{k=1}^d W_{kk} + {\alpha_v} \left( \left| v^- \otimes n - W\right| + \left| v^+\otimes n - W\right|\right)\,.$
\end{prpstn}
\begin{proof}
Repeating the first steps in the proof of Proposition
\ref{proposition:lowerboundsamesigns1d} one verifies  
\begin{align*}
 E[u,v,Q_n] &\geq {\alpha_F} \left|[u]\right| \int_{J_u} w \cdot \nu \text{ d}\mathcal{H}^d + {\alpha_v} \int_{Q_n}  \left| \nabla v \right| \text{ d}y + {\alpha_u} \left|[u]\right| 
 \mathcal{H}^d\left(J_u\right) \\
 &\geq {\alpha_F} \left|[u]\right| \int_{J_u} w \cdot \nu \text{ d}\mathcal{H}^d + {\alpha_v}  \left|\int_{Q^-_n} \nabla v \text{ d}y\right| + {\alpha_v} 
 \left|\int_{Q^+_n} \nabla v \text{ d}y\right| + {\alpha_u} \left|[u]\right| 
\\
 &= {\alpha_F} \left|[u]\right| \left( n_0 + \sum_{k=1}^d W_{kk}\right) + {\alpha_v} \left( \left| v^- \otimes n - W\right| + \left| v^+\otimes n - W\right|\right) + {\alpha_u} \left|[u]\right| \\
 &= {\alpha_F} \left|[u]\right| n_0 + {\alpha_u} \left|[u]\right| + f(W)\,.
 \end{align*}
Since we do not have a sharp truncation result in this case the 
following steps in the proof of Proposition \ref{proposition:lowerboundsamesigns1d} based on the functional $G$ do not extend to this situation.
\end{proof}

Now we proceed with the proofs of the main theorems.
\begin{proof}[Proof of Theorem \ref{theorem:bounds}] 
 First we treat the  case \ref{item:ofcsatisfied}, where the optical flow constraint can be satisfied locally.

 To prove the upper bound we first choose  $t\in [0,1]$ such that $(1,v^-t+v^+(1-t))\cdot n=0$. The upper bound follows then from Lemma \ref{lemmasubaddit}\ref{lemmasubaddit2}, taking $a=v^-t+v^+(1-t)$.
For the sake of illustration we give also an explicit construction here.  
  Let $\varphi:\R\to\R$ be a smooth monotone function such that 
  $\varphi(-\frac12)=0$,  $\varphi(0)=t$  and $\varphi(\frac12)=1$. 
  We set
  \begin{equation}
    u(y)=u^-+(u^+-u^-)\chi_{y\cdot n<0} \text{ and }
    v(y)=v^-+(v^+-v^-)\varphi(y\cdot n)
  \end{equation}
 and observe that  $(u,v) \in \mathcal{A}^C$. 
 A simple computation shows that  $K \leq {\alpha_u} \left|[u]\right| + {\alpha_v} \left|[v]\right|$.

To prove the lower bound we
 choose $(u,v) \in \mathcal{A}^C$ and verify using 
convexity and periodicity
\begin{equation*}
E[u,v,Q_n] \geq {\alpha_u} \left|[u] \right| + {\alpha_v} \left|[v] \otimes n\right| = {\alpha_u} \left|[u]\right| + {\alpha_v} \left|[v]\right|.
\end{equation*}
This concludes the proof of  \ref{item:ofcsatisfied}.

The upper bound in \ref{item:1dbounds}  
follows immediately from the construction
in Proposition \ref{propupperbound1d}.
The lower bound in \ref{item:1dbounds} was proven in  Proposition 
\ref{proposition:lowerboundsamesigns1d}.

Finally, to prove \ref{item:ddbounds} we observe that 
  Lemma \ref{lemmacomvex}  shows that
 if $n=\sum_{j=1}^l N^j\in S^d$ then
\begin{equation*}
K(u^+,u^-,v^+,v^-,n)=
  h(\sum_{j=1}^l N^j) \le 
\sum_{j=1}^l  h( N^j) = 
\sum_{j=1}^l  |N^j| K(u^+,u^-,v^+,v^-,\frac{N^j}{|N^j|})\,.
\end{equation*}
Using the bound from  Lemma \ref{lemmasubaddit}\ref{lemmasubaddit2}
in each term in the sum gives the assertion.
\end{proof}

\begin{proof}[Proof of Theorem \ref{theorem:sufficient}] 

\ref{theorem:sufficientit1}: The  first assertion
immediately follows from Theorem \ref{theorem:bounds} \ref{item:ofcsatisfied}
and its proof. 

 \ref{item:theoremsufficient1d}: 
By Proposition \ref{proposition:lowerboundsamesigns1d} we have 
$K\ge \min\underbar{K}(\R^2)$, where
\begin{equation*}
\underbar{K}(N^+)= \left({\alpha_u} \left|[u]\right| + {\alpha_v} \left|[v]\right| \right)  \left( \left|N^+\right|+ \left| n-N^+\right| \right) + 
{\alpha_F} \left|[u]\right| \, | w^-\cdot n + [v]N^+_1| \,.
\end{equation*} 
We first show that if  $[v](w^-\cdot n)n_1> 0$ then
$N^+=0$ is a minimizer of $\underbar{K}$.
Since $N^+\mapsto |n-N^+|$ is convex, it lies above its tangent,
$|n-N^+|\ge 1-n\cdot N^+$. Using $\left|w^- \cdot n + [v]N_1^+\right| \geq \left| w^- \cdot n\right| + \operatorname{sign}(w^-\cdot n) [v] N_1^+$ for the last term we obtain 
\begin{alignat*}1
  \underbar{K}(N^+)-  \underbar{K}(0)&\ge  \left({\alpha_u} \left|[u]\right| + {\alpha_v} \left|[v]\right| \right)  
\left( \left|N^+\right|- n\cdot N^+ \right) + 
{\alpha_F} \left|[u]\right| \operatorname{sign}(w^-\cdot n) [v] N^+_1\,.
\end{alignat*}
Using the fact that $[v](w^-\cdot n)n_1$ is positive we obtain
\begin{equation*}
 \frac{\underbar{K}(N^+)-  \underbar{K}(0)}
{{\alpha_u} |[u]| + {\alpha_v}  |[v]|}
\ge  \left|N^+\right|- n\cdot N^+ + \xi n_1 N^+_1\,,
\end{equation*}
where
\begin{equation*}
\xi=\frac{{\alpha_F} |[u]|  |[v]|}{({\alpha_u} |[u]| + {\alpha_v}  |[v]|) |n_1|}\,.
\end{equation*}
Since by assumption $\xi\in[0,2]$ we have
\begin{equation*}
n\cdot N^+-\xi n_1N^+_1=(1-\xi) n_1N^+_1+n_0N^+_0\le |n|\,|N^+|=|N^+|\,,
\end{equation*}
therefore $\underbar{K}(N^+)\ge  \underbar{K}(0)$.

Recalling the upper bound we have $\underline K(0)\le K\le \overline K(0)$,
and since  $\overline K(0)=\underline K(0)$ we conclude that
in the case $[v](w^-\cdot n)n_1> 0$
one has $K=\overline K(0)= {\alpha_u} \left|[u]\right| + {\alpha_v} \left|[v]\right|   +{\alpha_F} \left|[u]\right| |w^-\cdot n|$,
with a simple profile since we have set $N^+=0$ in the construction of the upper bound. The same result still holds
if $[v]=0$, as a simple inspection of the lower bound shows.

Swapping $(u^+,v^+)$ with $(u^-,v^-)$ and using  Lemma \ref{lemma: invariance} we obtain that
if   $[v](w^+\cdot n)n_1< 0$ then $K=\overline K(n)= {\alpha_u} \left|[u]\right| + {\alpha_v} \left|[v]\right|   +{\alpha_F} \left|[u]\right| |w^+\cdot n|\,, $
again with a simple profile. For $[v]=0$ the two assertions coincide.
Since for  $n_1=0$ our assumption $\xi\le2$ is only satisfied in the trivial 
case $[u]=0$ and $[v]=0$, we have shown that if  $w^+\cdot n$ and $w^-\cdot n$ are nonzero and have the same sign 
the profile is simple and the interfacial energy is either $\overline K(0)$ or 
$\overline K(n)$; since they are both admissible the energy necessarily  
is  the minimum of the two.

\ref{item:theoremsufficient}: 
At first we observe that Lemma \ref{lemmasubaddit} \ref{lemmasubaddit2} implies that $\alpha_F \min\{ |w^- \cdot n|, |w^+ \cdot n| \} + \alpha_v |[v]| + \alpha_u |[u]|$ is an upper bound.
Swapping $n$ with $-n$ we can assume that $w^+\cdot n$ and $w^-\cdot n$ are strictly positive,
swapping $(u^+,v^+)$ with $(u^-,v^-)$ we can assume that
\begin{equation*}
  0<w^-\cdot n\le w^+\cdot n\,.
\end{equation*}
From Proposition \ref{proposition:lowerbounddge1} we know that 
$E[u,v,Q_n] \geq {\alpha_F} \left|[u]\right| n_0 + {\alpha_u} \left|[u]\right| + f(W)$ with 
\begin{equation*}
 f(W) = {\alpha_F} \left|[u]\right| \sum_{k=1}^d W_{kk} + {\alpha_v} \left( \left| v^- \otimes n - W\right| + \left| v^+\otimes n - W\right|\right)\,.
\end{equation*}
Notice that $$f\left(v^- \otimes n\right) = {\alpha_F} \left|[u]\right|\sum_{k=1}^d v^-_k n_k + {\alpha_v} \left|[v]\right|\,.$$
 As $w^+\cdot n \ge w^-\cdot n\ge 0$
it suffices to show that $f$ has a minimum at $v^- \otimes n$. Since $f$ is convex this is the case if and only if $0$ is a subgradient at $v^- \otimes n$.
 It can be easily seen that the set of subgradients of $f$ at $v^- \otimes n$ is
 \begin{equation*}
  \partial f\left(v^- \otimes n\right) = {\alpha_F} \left|[u]\right| \left(\delta_{ij}\right)_{\substack{1\leq i \leq d, \\ 0 \leq j \leq d}} 
  + {\alpha_v} \left( \overbar{B_1(0)} - \frac{[v] \otimes n}{\left|[v]\right|}\right).
 \end{equation*}
Hence $0 \in \partial f\left(v^- \otimes n\right)$ if and only if
\begin{equation*}
 \frac{{\alpha_F}}{{\alpha_v}} \left|[u]\right| \left(\delta_{ij}\right)_{\substack{1\leq i \leq d, \\ 0 \leq j \leq d}} - \frac{[v] \otimes n}{\left|[v]\right|} \in \overbar{B_1(0)}\subset\R^{d(d+1)}
\end{equation*}
(we recall that we are using the Euclidean norm on $\R^{d(d+1)}$).
Squaring the left hand side shows that this is equivalent to
\begin{equation*}
  \left(\frac{{\alpha_F}}{{\alpha_v}}\right)^2 \left|[u]\right|^2 d - 2 \frac{{\alpha_F}}{{\alpha_v}} \left|[u]\right| \frac{[w] \cdot n}{\left|[v]\right|} + 1\le 1
\end{equation*}
which in turn is the same as
\begin{equation*}
 \frac{{\alpha_F}}{{\alpha_v}} \left|[u]\right| d - 2 \frac{[w] \cdot n}{\left|[v]\right|} \le 0.
\end{equation*}
This holds since we are in the case $[w]\cdot n\ge 0$ and
by assumption   $2{\alpha_v}  |[w] \cdot n| \geq {\alpha_F} d  |[u]|  |[v]|$.

We finally prove \ref{item:simpletok}. The upper bound follows from Theorem \ref{theorem:bounds}\ref{item:ddbounds}  with $l=1$. To prove the lower bound, we observe that the construction in Lemma 
\ref{lem:coarea} does not modify the property of being one-dimensional, hence we can assume that $u,v\in \mathcal{A}^C$ are of the form 
$u=u^-+[u]\chi_{\omega} (y\cdot n)$, $v=\tilde v(y\cdot n)$, for some set of finite perimeter  $\omega\subset \R^{d+1}$ and some
function $\tilde v:\R \rightarrow \R^d$. Then,
setting  $t_*\in J_u=\partial\omega$,
\begin{equation*}
  E[u,v,Q_n] \ge {\alpha_u} |[u]| + {\alpha_v} \int_{-1/2}^{1/2} |\tilde v'|(t)  dt
  + {\alpha_F} |[u]| |(1,\tilde v(t_*))\cdot n|\,.
\end{equation*}
With $a=\tilde v(t_*)$ the assertion follows.
\end{proof}

\subsection{Illustration of expected microstructures for $d\geq 1$}\label{sec:illustration}
 
 \begin{figure}[h!,t]
\begin{tikzpicture}[scale=0.7]
\draw[->] (0,0,0) -- (0,4,0);
\draw[->] (0,0,0) -- (4,0,0);
\draw[->] (0,0,0) -- (0,0,-3);

\draw (0,0,-3) node[black, anchor=south]{$y$};
\draw (0,4,0) node[black, anchor=south]{$t$};
\draw (4,0,0) node[black, anchor=north]{$x$};

\draw[->] (2,2-4/7+3/28,-3/2) -- (2.5,2-4/7+3/28+1.75,-3/2+0.25);
\draw (2.5,2-4/7+3/28+1.75,-3/2+0.25) node[black,anchor=west]{$n$};

\fill[black,semitransparent] (0,2,0) -- (0,2+6/28,-3) -- (4,  2 + 6/28 - 8/7 ,-3) -- (4, 2 - 8/7 ,0) -- (0,2,0);
\draw (0,0,-3) node[black, anchor=south]{$y$};

\fill[gray!30,semitransparent] (0,2,0) -- (4,2 - 8/7 ,0) -- (4,4,0) -- (0,4,0) -- (0,2,0);
\fill[gray!30,semitransparent] (4,2 - 8/7 ,0) -- (4,  2 + 6/28 - 8/7 ,-3) -- (4,4,-3) -- (4,4,0) -- (4,2 - 8/7 ,0);
\fill[gray!30,semitransparent] (0,4,0) -- (0,4,-3) -- (4,4,-3) -- (4,4,0) -- (0,4,0);

\fill[gray!70,semitransparent] (4,2 - 8/7 ,0) -- (0,2,0) -- (0,0,0) -- (4,0,0) -- (4,2 - 8/7 ,0);
\fill[gray!70,semitransparent] (4,2 - 8/7 ,0) -- (4,  2 + 6/28 - 8/7 ,-3) -- (4,0,-3) -- (4,0,0) -- (4,  2 - 8/7 ,0);

\draw[dashed] (0,4,0) -- (4,4,0) -- (4,4,-3) -- (0,4,-3) -- (0,4,0);
\draw[dashed] (4,0,0) -- (4,0,-3) -- (4,4,-3) -- (4,4,0) -- (4,0,0);

\end{tikzpicture} \quad \quad
\includegraphics[width=0.27\textwidth]{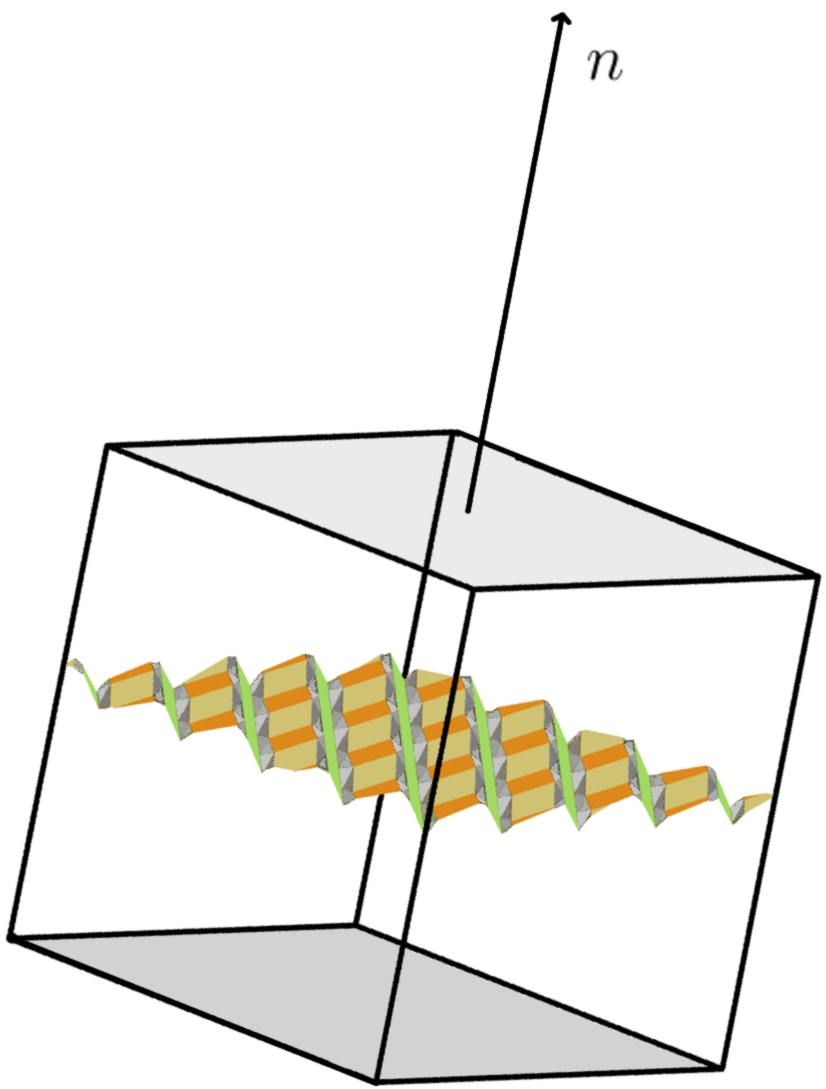} \quad \quad
\includegraphics[width=0.27\textwidth]{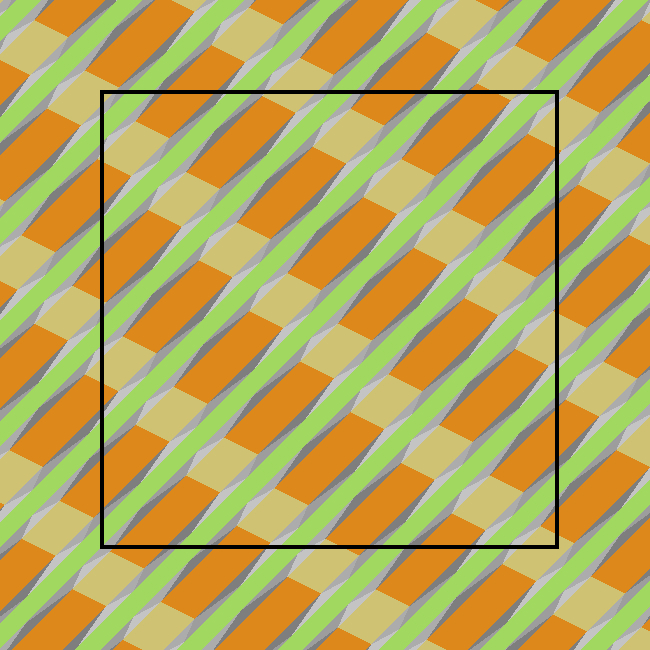} 
\caption{Sketch of a microscopic pattern corresponding to an explicit construction of the upper bound of $K$ for $d=2$ (left: macroscopic configuration, middle: microscopic pattern, right: top view of the pattern)\label{fig:pattern}}
\end{figure}
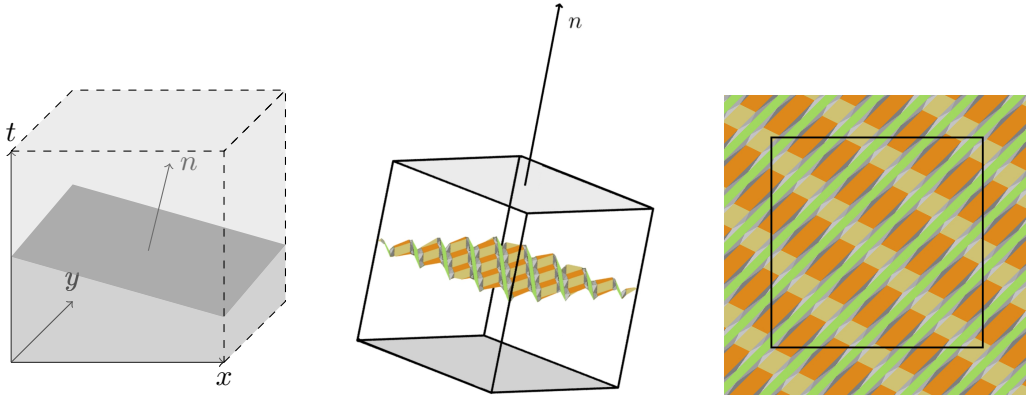
 
 In the higher dimensional case ($d>1$) the upper bound in Theorem \ref{theorem:bounds} (iii) refers to a set of vectors 
 ($N^1,\ldots, N^l$). In the case $d=1$ the corresponding vectors $N^+$ and $N^-$ explicitly appear in a geometric construction of a microscopic pattern in the 
 proof of Proposition \ref{propupperbound1d}.  Using the convexity of the microscopic energy $K$ stated in Lemma \ref{lemmacomvex} this explicit construction is 
 not required in the proof of Theorem \ref{theorem:bounds} (iii). Nevertheless, explicit microscopic patterns can be constructed, where the vectors $N^j$ are weighted normals 
 with $\frac{N^j}{|N^j|}$ being the normal on a set of interface facets of total area $|N^j|$.  Figure \ref{fig:pattern} sketches for $d=2$ such a microscopic pattern with three
 weighted normals 
 $$
 Q \left(\begin{array}{c} \frac14 \\ -\frac14 \\ \frac14 \end{array}\right),\; 
 Q \left(\begin{array}{c} -\frac12 \\ -\frac12 \\ \frac12 \end{array}\right),\; 
 Q \left(\begin{array}{c} \frac14 \\ \frac34 \\ \frac14 \end{array}\right)\,,
$$
where $Q$ is a suitable rotation in $\R^3$. The resulting effective interface normal $n=N^1+N^2+N^3= Q (0,0,1)^T$ is the macroscopic normal on the jump set $J_u$. The underlying pattern is based on nested lamination, i.e.~a lamination pattern of facets perpendicular to $N^2$ and $N^3$ (plotted in light and dark orange) is altered with facets perpendicular to $N^1$ (plotted in green). To compensate the lack of rank-$1$ consistency a thin transition pattern is introduced in between (plotted in grey).

\bibliographystyle{plain}

\bibliography{opticalflow.bib}

\end{document}